\documentclass[10pt,reqno,a4paper]{amsart}

\usepackage{amsmath,amssymb,amsthm,mathtools}
\usepackage{bm}
\usepackage{mathrsfs}
\usepackage{enumitem}
\usepackage{microtype}
\usepackage{a4wide}
\usepackage{indentfirst}
\usepackage[numbers,sort&compress]{natbib}
\usepackage[colorlinks=true,linkcolor=blue,citecolor=blue,urlcolor=blue]{hyperref}

\numberwithin{equation}{section}
\newtheorem{theorem}{Theorem}[section]
\newtheorem{proposition}[theorem]{Proposition}
\newtheorem{lemma}[theorem]{Lemma}
\newtheorem{corollary}[theorem]{Corollary}
\theoremstyle{remark}
\newtheorem{remark}[theorem]{Remark}

\newcommand{\Sn}{\mathbb S^n}
\newcommand{\grad}{\nabla_g}
\newcommand{\Hess}{\nabla_g^2}
\newcommand{\Lap}{\Delta_g}
\newcommand{\II}{\mathrm{II}}
\newcommand{\dist}{\operatorname{dist}_g}
\newcommand{\R}{\mathbb R}

\newcommand{\tr}{\operatorname{tr}}
\newcommand{\rank}{\operatorname{rank}}
\newcommand{\Ker}{\operatorname{Ker}}

\title[Lane--Emden and Torsion on the Sphere]{Transformed Convexity for Lane--Emden Equations and Torsion on Spherical Domains}
\author{Xiuda Liang}
\address[Xiuda Liang]{School of Mathematics and Statistics, Central China Normal University, Wuhan 430079, China}
\email{lxddd@mails.ccnu.edu.cn}

\author{Wenjie Wang}
\address[Wenjie Wang]{School of Mathematics and Statistics, Central China Normal University, Wuhan 430079, China}
\email{wjwang3269@mails.ccnu.edu.cn}

\begin{document}
\begin{abstract}
We study the geometry of positive Dirichlet solutions on smooth,
uniformly geodesically convex domains $\Omega\subsetneq\mathbb S^n$,
$n\ge2$. Set $p_n^*=2n/(2n-3)$. For the Lane--Emden equation
$-\Delta_g u=u^p$, $1<p\le p_n^*$, we prove that the set of positive
solution pairs $(p,u)$ has a connected component containing the
near-linear solution curve and projecting onto $(1,p_n^*]$.
Every pair on this component satisfies
$\nabla_g^2(u^{-(p-1)/3})>0$ in $\Omega$.
For the torsion equation $-\Delta_g u=1$, we establish
$\nabla_g^2\sqrt{u}<0$ for its unique Dirichlet solution in all
dimensions $n\ge2$, without an additional diameter or circumradius
assumption. In both settings, the solutions specified above have a
unique nondegenerate maximum and strictly geodesically convex
positive superlevel sets. For the non-power reaction
$f(s)=s^p/(1+\lambda s^\delta)$, with $1<p\le p_n^*$, $\lambda>0$,
and $0<\delta\le(p-1)/3$, we also prove the existence of at least
one positive solution $u$ for which the transform
$\left(\int_u^\infty ds/f(s)\right)^{1/3}$ has positive definite
Hessian. The proofs combine a spherical Schur-complement estimate
with a priori bounds and continuation.
\end{abstract}

\keywords{Lane--Emden equation, spherical torsion, strict convexity, geodesically convex domains, Schur complement}
\subjclass[2020]{35J60, 35B50, 53C21}
\maketitle
\enlargethispage{5pt}

\section{Introduction and main results}
\label{sec:intro}

Let $(\mathbb S^n,g)$ be the unit round sphere, with $n\ge2$, and
let $\Omega\subsetneq\mathbb S^n$ be a smooth domain. Our starting
point is the Dirichlet problem
\begin{equation}\label{eq:introGeneral}
\begin{cases}
-\Lap u=f(u) & \text{in }\Omega,\\
u>0 & \text{in }\Omega,\\
u=0 & \text{on }\partial\Omega,
\end{cases}
\end{equation}
Our principal examples are the Lane--Emden reaction $f(u)=u^p$, $p>1$,
and the torsion problem $f\equiv1$. We seek geometric information that
is generally inaccessible from the maximum principle alone: uniqueness
and nondegeneracy of the maximum, convexity of positive superlevel sets,
and definiteness of the Riemannian Hessian of a suitable transform.
More precisely, a decreasing transform $v=\Psi(u)$ with positive
definite Hessian has strictly convex sublevel sets, which are the
positive superlevel sets of $u$. In the situations considered below,
the behavior at the boundary also ensures that $v$ has a unique
minimum, corresponding to a nondegenerate maximum of $u$. Identifying
transforms for which this mechanism works is our main objective.

We write $\Delta_g$, $\nabla_g$, and
$\nabla_g^2$ for the Laplace--Beltrami operator, gradient,
and Hessian. A domain $\Omega$ is called
\emph{uniformly geodesically convex} if $\overline\Omega$ is
geodesically convex and its boundary satisfies
\begin{equation}\label{eq:introUniformConvexity}
\II(X,Y)=g(\nabla_X\nu,Y),\qquad
\II\ge\kappa_0g_{\partial\Omega}
\quad\text{for some }\kappa_0>0,
\end{equation}
where $\nu$ is the outward unit normal. Here convexity refers to
shorter minimizing geodesics between nonantipodal points. This is the
geometric hypothesis in the main theorems; we impose no separate bound
on the diameter or circumradius of $\Omega$.

Classical results in Euclidean space established several complementary
forms of concavity. Brascamp and Lieb
\cite{BrascampLieb1976} established a foundational log-concavity
framework, while Makar--Limanov \cite{MakarLimanov1971} proved the
square-root concavity of planar torsion functions. Kennington
\cite{Kennington1985} developed power-concavity results for elliptic
boundary value problems. Convexity principles for semilinear equations
were subsequently investigated by Korevaar \cite{Korevaar1983},
Caffarelli and Spruck \cite{CaffarelliSpruck1982}, and Caffarelli and
Friedman \cite{CaffarelliFriedman1985}. A related approach analyzes
the rank of the Hessian; see Korevaar--Lewis
\cite{KorevaarLewis1987}, Caffarelli--Guan--Ma
\cite{CaffarelliGuanMa2007}, and Bian--Guan \cite{BianGuan2009}.
For the geometry of critical points of semistable solutions in convex
Euclidean domains, see \cite{CabreChanillo1998}. Convexity of the
domain alone does not, however, ensure convexity of all superlevel
sets for general semilinear reactions; see \cite{HamelNadirashviliSire2016}.
Further semilinear concavity criteria can be found in
\cite{ChauWeinkove2023}. These works motivate the search for a strict
convexity mechanism that remains effective in the presence of ambient
curvature.

The situation on positively curved manifolds is more delicate.
Commuting covariant derivatives introduces curvature terms into the
evolution or elliptic identities for the Hessian, so Euclidean
convexity arguments need not carry over unchanged. For the first
Dirichlet eigenfunction of a strictly convex spherical domain,
Lee and Wang \cite{LeeWang1987} established spherical
log-concavity; the strict formulation used near the linear limit
is stated in \cite[Proposition~2.1]{GrossiProvenzanoRaom2026}.
Langford and Scheuer \cite{LangfordScheuer2021} proved a concavity
principle for degenerate elliptic equations on spherical domains.
Aryan and Law \cite{AryanLaw2025} extended this circle of ideas to
locally symmetric spaces with nonnegative sectional curvature.
Using conformal deformations, Khan, Saha, and Tuerkoen
\cite{KhanSahaTuerkoen2025} established concavity estimates that
also apply to spherical torsion under a domain-size restriction.
Related symmetry and power-concavity results in Riemannian settings
can be found in \cite{KumaresanPrajapat1998,IshigeSalaniTakatsu2022};
for critical-point properties of semistable solutions on convex
constant-curvature surfaces, see \cite{GrossiProvenzano2024}.
These contributions provide distinct, complementary approaches to
curvature-dependent convexity problems.

Of particular relevance is the recent work of Grossi, Provenzano,
and Raom \cite{GrossiProvenzanoRaom2026} on uniformly convex domains
of $\mathbb S^2$. For $1<p\le3$, they proved the strict convexity of
$u^{-(p-1)/2}$ for \emph{every} positive Lane--Emden solution; their
results also include strict square-root concavity for planar spherical
torsion. Gladiali, Grossi, and Provenzano
\cite{GladialiGrossiProvenzano2026}, in a different direction,
established uniqueness and nondegeneracy of the critical point for
positive semistable solutions on certain spherical and hyperbolic
surfaces under sharp geometric conditions. The latter work concerns
the critical-point structure of semistable solutions, rather than
positive definiteness of the Hessian of the transforms studied here.

The torsion problem deserves an additional qualification regarding
prior work. In Euclidean domains, quantitative estimates beyond the
classical concavity statements were obtained, for instance, by
Jia, Ma, and Shi \cite{JiaMaShi2023}. For the sphere, Korevaar
\cite{Korevaar1987} referred to unpublished work with Treibergs
on square-root concavity; this historical reference is also discussed
in \cite{KhanSahaTuerkoen2025}. Accordingly, we do not claim that the
spherical concavity phenomenon was previously unknown. Our torsion
result concerns a strict Hessian inequality in all dimensions for
smooth uniformly geodesically convex domains, with no extra size
restriction, and a quantitative consequence. In particular, the
higher-dimensional proof must control all possible null directions
of the Hessian, rather than rely on arguments specific to surfaces.
Recent related developments include concavity principles for
elliptic and parabolic equations on Euclidean convex domains
\cite{GalloMoraschiSquassina2026} and square-root concavity for
torsion on small horo-convex domains in hyperbolic space
\cite{ZhangZhou2026}. The ambient geometry and the hypotheses in those
works differ from the spherical setting considered here.

The two problems also place different demands on global arguments.
For Lane--Emden, the exponent varies and the result concerns a
connected component of the full positive-solution set; no uniqueness
statement is made for arbitrary exponents. This differs both in the
transform and in its quantifiers from
\cite{GrossiProvenzanoRaom2026}. For torsion, the solution is unique
on each fixed domain, but strict concavity must persist as the domain
is deformed. The common analytical tool is a local Schur-complement
estimate that rules out null directions once the transformed Hessian
is known to be nonnegative. Compactness and boundary control are
needed to make this local conclusion effective globally.

We now give the precise statements, beginning with the Lane--Emden
equation
\begin{equation}\label{eq:mainLE}
\begin{cases}
-\Lap u=u^p & \text{in }\Omega,\\
u>0 & \text{in }\Omega,\\
u=0 & \text{on }\partial\Omega,
\end{cases}
\qquad p>1.
\end{equation}
No symmetry assumption on $\Omega$ is required.

\begin{theorem}\label{thm:main}
Let $n\ge2$, and let $\Omega\subsetneq\Sn$ be a smooth, connected,
uniformly geodesically convex domain. Define
\[
P:=p_n^*:=\frac{2n}{2n-3}=1+\frac{3}{2n-3},
\qquad X:=C_0(\overline\Omega),
\]
and let
\[
\mathscr S:=\bigl\{(p,u)\in(1,P]\times X:
 u\text{ is a positive classical solution of }\eqref{eq:mainLE}\bigr\}
\]
be endowed with the relative topology of $(1,P]\times X$.
Then the following assertions hold.

\begin{enumerate}[label=\textup{(\roman*)},leftmargin=2.9em]
\item \emph{Near-linear solutions.}
There exists $\varepsilon_\Omega>0$ such that, for every
$1<p<1+\varepsilon_\Omega$, problem~\eqref{eq:mainLE}
has a unique positive solution $u_p$. This solution is nondegenerate and satisfies
\[
\Hess\bigl(u_p^{-(p-1)/3}\bigr)>0\qquad\text{in }\Omega.
\]
The near-linear solution curve
\[
\Gamma:=\bigl\{(p,u_p):1<p<1+\varepsilon_\Omega\bigr\}
\]
is connected.

\item \emph{Global connected component.}
Let $\mathscr K$ denote the connected component of $\mathscr S$
containing $\Gamma$. Then
\begin{equation}\label{eq:componentProjection}
\pi_p(\mathscr K)=(1,P],
\end{equation}
and every $(p,u)\in\mathscr K$ satisfies
\begin{equation}\label{eq:mainconv}
\Hess\bigl(u^{-(p-1)/3}\bigr)>0\qquad\text{in }\Omega.
\end{equation}
Moreover, $\mathscr K\cap([a,P]\times X)$ is compact for every
$a\in(1,P]$.

\item \emph{Geometric consequences.}
Every $u$ with $(p,u)\in\mathscr K$ has a unique critical point,
which is a nondegenerate strict maximum. For every
$0<t<\max_{\Omega}u$, its positive superlevel set
$\{u>t\}$ is strictly geodesically convex.
\end{enumerate}
\end{theorem}

For example, $p_2^*=4$, $p_3^*=2$, and $p_4^*=8/5$.
Thus, in dimension two, the admissible interval includes $3<p\le4$,
which lies beyond the exponent range of
\cite{GrossiProvenzanoRaom2026}. This comparison concerns different
transforms and different sets of solutions: the cited paper treats
every positive solution for $1<p\le3$, whereas our conclusion holds
on the component $\mathscr K$. In particular, Theorem~\ref{thm:main}
does not imply uniqueness at a fixed exponent or convexity for other
components. The endpoint $p_n^*$ arises from our Hessian criterion;
we do not assert its optimality for the Lane--Emden equation.

Our second principal result concerns the torsion function, the
unique solution of
\begin{equation}\label{eq:torsion}
\begin{cases}
-\Lap u=1 & \text{in }\Omega,\\
u=0 & \text{on }\partial\Omega.
\end{cases}
\end{equation}
In this case the relevant concave transform is $\sqrt u$.

\begin{theorem}\label{thm:torsion}
Let $n\ge2$, and let $\Omega\subsetneq\Sn$ be a smooth uniformly
geodesically convex domain. If $u$ solves \eqref{eq:torsion}, then
\begin{equation}\label{eq:torsionconc}
\Hess\sqrt u(X,X)<0\quad
\text{for every }x\in\Omega\text{ and }0\ne X\in T_x\Sn.
\end{equation}
No additional diameter or circumradius bound on $\Omega$ is required.
\end{theorem}

The strict Hessian inequality has the following geometric and
quantitative consequences.

\begin{corollary}\label{cor:torsionConsequences}
Under the assumptions of Theorem~\ref{thm:torsion}, the torsion function
has exactly one critical point, which is a nondegenerate maximum. Each
positive superlevel set $\{u>c\}$, $0<c<\max_\Omega u$, is a relatively
compact smooth geodesically convex domain whose outward second
fundamental form is positive definite. In addition, there exists
$c_\Omega>0$ such that
\begin{equation}\label{eq:torsionUniform}
\Hess\sqrt u\le-c_\Omega g\qquad\text{in }\Omega.
\end{equation}
\end{corollary}

The proof strategies have a common local component but distinct
global steps. For Lane--Emden, blow-up estimates and Leray--Schauder
continuation produce the connected solution set, in the spirit of
\cite{Rabinowitz1971,Mawhin1997}. For torsion, the argument proceeds
by deforming a geodesic ball through uniformly geodesically convex
domains. In both cases, boundary estimates must be uniform along the
chosen family: otherwise a null Hessian direction could arise at the
boundary before the interior strictification argument applies. We
establish these estimates in the relevant proof sections.

The approach also yields a nonhomogeneous extension. We state one
explicit consequence for the inhibited reaction
$f(u)=u^p/(1+\lambda u^\delta)$; the more general intrinsic criterion
is developed in Section~\ref{sec:nonpower}.

\begin{theorem}
\label{thm:inhibited}
Let $\Omega\subset\Sn$ satisfy the geometric assumptions of
Theorem~\ref{thm:main}. Suppose that
\[
1<p\le p_n^*:=\frac{2n}{2n-3},\qquad \lambda>0,\qquad
0<\delta\le\frac{p-1}{3}.
\]
Then
\begin{equation}\label{eq:inhibitedPDE}
\begin{cases}
-\Lap u=\dfrac{u^p}{1+\lambda u^\delta}&\text{in }\Omega,\\
u>0&\text{in }\Omega,\\
u=0&\text{on }\partial\Omega
\end{cases}
\end{equation}
admits at least one positive classical solution satisfying
\begin{equation}\label{eq:inhibitedConvIntro}
\Hess\left[
\left(\frac{u^{1-p}}{p-1}
+\frac{\lambda u^{1-p+\delta}}{p-1-\delta}\right)^{1/3}\right]>0
\qquad\text{in }\Omega.
\end{equation}
This solution has a unique nondegenerate maximum and strictly
geodesically convex positive superlevel sets; in dimension $n=2$ the
assertion includes the interval $3<p\le4$.
No assertion is made for all positive solutions of \eqref{eq:inhibitedPDE}.
\end{theorem}

For $\lambda>0$, this reaction is not a pure power. It provides an
explicit nonhomogeneous application without introducing an additional
principal theme.

The remainder of the paper is organized as follows.
Section~\ref{sec:local} establishes the spherical Hessian identities
and the local strictification estimate. Section~\ref{sec:LEnear}
proves the Lane--Emden theorem, from the near-linear analysis to the
global continuation argument. Section~\ref{sec:torsionglobal}
develops the domain deformation and proves the torsion theorem.
Section~\ref{sec:nonpower} treats the intrinsic non-power criterion
and the explicit inhibited reaction. Uniform compactness and domain
variation estimates are collected in Appendix~\ref{app:technical};
coordinate calculations appear in Appendix~\ref{app:coordinates}.

\section{Spherical Hessian identities and strictification}
\label{sec:local}

\subsection{Covariant identities and the Schur-complement estimate}

Our curvature convention is
\begin{equation}
\label{eq:curvconv}
R(X,Y)Z=\nabla_X\nabla_YZ-\nabla_Y\nabla_XZ-\nabla_{[X,Y]}Z.
\end{equation}
On the unit sphere,
\begin{equation}
\label{eq:spherecurv}
R(X,Y)Z=\langle Y,Z\rangle X-\langle X,Z\rangle Y.
\end{equation}
For a covariant tensor $T$, its rough Laplacian is
\[
\Delta^\nabla T=\sum_{k=1}^n\left(\nabla_{e_k}\nabla_{e_k}T-\nabla_{\nabla_{e_k}e_k}T\right)
\]
in a local orthonormal frame.

\begin{lemma}
\label{lem:sphereidentities}
Let $v\in C^4(U)$ on an open set $U\subset\Sn$.  Set
\[
A=\Hess v,\qquad q=|\grad v|^2,\qquad F=\Lap v.
\]
In a local orthonormal frame, at the point of evaluation,
\begin{align}
\nabla_kA_{ij}-\nabla_iA_{kj}
&=v_i\delta_{kj}-v_k\delta_{ij},
\label{eq:CodazziSphere}\\
(\Delta^\nabla A)_{ij}
&=\nabla_i\nabla_jF+2nA_{ij}-2F\delta_{ij},
\label{eq:HessLapSphere}\\
\nabla_i\nabla_jq
&=2\sum_kA_{ik}A_{jk}
+2\sum_kv_k\nabla_kA_{ij}
+2(q\delta_{ij}-v_iv_j).
\label{eq:qSphere}
\end{align}
\end{lemma}

\begin{proof}
The curvature identity for the covector $dv$ yields \eqref{eq:CodazziSphere}.  Contracting this identity gives
\[
\sum_k\nabla_kA_{kj}=\nabla_jF+(n-1)v_j.
\]
The curvature action on the two covariant slots of $A$ gives
\[
\sum_k(\nabla_k\nabla_i-\nabla_i\nabla_k)A_{kj}=nA_{ij}-F\delta_{ij}.
\]
Applying $\nabla_k$ to \eqref{eq:CodazziSphere} and summing gives
\[
(\Delta^\nabla A)_{ij}
=\nabla_i\nabla_jF+(n-1)A_{ij}+nA_{ij}-F\delta_{ij}+A_{ij}-F\delta_{ij},
\]
as asserted in \eqref{eq:HessLapSphere}. Finally,
\[
\nabla_iq=2\sum_kv_kA_{ik}.
\]
Differentiating once more and using
\[
\nabla_jA_{ik}=\nabla_kA_{ij}+v_k\delta_{ij}-v_j\delta_{ik}
\]
gives \eqref{eq:qSphere}.
\end{proof}

\begin{lemma}
\label{lem:SchurMaster}
Let $v\in C^4(U)$ satisfy
\[
\Hess v\ge0,
\qquad
F:=\Lap v=\Phi(v,q)>0,
\qquad q:=|\grad v|^2,
\]
where $\Phi\in C^2$.  Suppose that $\Hess v$ has a null direction at $x_0\in U$, and write
\[
r:=\rank\Hess v(x_0),\qquad 1\le r\le n-1.
\]
Then there is a $C^2$ nonnegative Schur-complement test function $\phi$ with $\phi(x_0)=0$ such that
\begin{equation}
\label{eq:MasterSchur}
\Delta\phi(x_0)
\le
\left[
\Phi_{vv}-2\Phi_q
-\frac{2}{F}(r+\Phi_v)^2
\right]v_1^2
+2\Phi_q q-2F.
\end{equation}
All derivatives of $\Phi$ are evaluated at $(v(x_0),q(x_0))$.
\end{lemma}

\begin{proof}
Choose an orthonormal basis at $x_0$ which diagonalizes $A:=\Hess v(x_0)$ and whose first vector $e_1$ is null.  Extend it to a smooth local orthonormal frame with $\nabla e_i(x_0)=0$.  Let $G$ be the set of positive eigenvalue indices and write
\[
a_\alpha=A_{\alpha\alpha}(x_0)>0\quad(\alpha\in G),
\qquad
\sum_{\alpha\in G}a_\alpha=F(x_0),
\qquad |G|=r.
\]
Since $A\ge0$, $A_{1j}(x_0)=0$ for every $j$.  Moreover, for every null diagonal direction $i$, the scalar function $A(e_i,e_i)$ has a local minimum zero at $x_0$, and hence
\begin{equation}
\label{eq:masterNullDeriv}
\nabla_kA_{ii}(x_0)=0\qquad\text{for all }k.
\end{equation}

In a neighborhood of $x_0$, let
\[
M=(A_{\alpha\beta})_{\alpha,\beta\in G},
\qquad z=(A_{\alpha1})_{\alpha\in G}.
\]
After shrinking the neighborhood, $M$ is positive definite.  Put
\[
b=M^{-1}z,
\qquad
V=e_1-\sum_{\alpha\in G}b_\alpha e_\alpha,
\qquad
\phi=A(V,V)=A_{11}-z^TM^{-1}z.
\]
Because $A\ge0$, one has $\phi\ge0$, while $\phi(x_0)=0$.  Hence
\begin{equation}
\label{eq:masterMin}
\Delta\phi(x_0)\ge0.
\end{equation}
At $x_0$, $b=0$, $V=e_1$, and
\[
\nabla_kV=-\sum_{\alpha\in G}\frac{\nabla_kA_{\alpha1}}{a_\alpha}e_\alpha.
\]
The tensor product rule gives
\[
\begin{aligned}
\Delta(A(V,V))={}&(\Delta^\nabla A)(V,V)
+4\sum_k(\nabla_kA)(\nabla_kV,V)\\
&+2\sum_kA(\nabla_kV,\nabla_kV)
+2A(\Delta^\nabla V,V).
\end{aligned}
\]
The last term vanishes at $x_0$ because $A(e_1,\cdot)=0$, and the middle two terms combine to yield the exact identity
\begin{equation}
\label{eq:MasterSchurExact}
\Delta\phi(x_0)
=(\Delta^\nabla A)_{11}
-2\sum_{\alpha\in G}\sum_k\frac{(\nabla_kA_{1\alpha})^2}{a_\alpha}.
\end{equation}
In particular, this computation does not require differentiating eigenvectors or assuming locally constant rank.

At the null direction, Lemma~\ref{lem:sphereidentities} and \eqref{eq:masterNullDeriv} give
\[
q_1=0,
\qquad
q_{11}=2(q-v_1^2).
\]
Since $v_{11}=A_{11}=0$,
\[
F_1=\Phi_vv_1,
\qquad
F_{11}=\Phi_{vv}v_1^2+2\Phi_q(q-v_1^2).
\]
Therefore \eqref{eq:HessLapSphere} gives
\begin{equation}
\label{eq:MasterDeltaA}
(\Delta^\nabla A)_{11}
=(\Phi_{vv}-2\Phi_q)v_1^2+2\Phi_q q-2F.
\end{equation}
On the other hand, \eqref{eq:CodazziSphere} implies
\[
\nabla_\alpha A_{1\alpha}=\nabla_1A_{\alpha\alpha}+v_1.
\]
Summing over $\alpha\in G$, using \eqref{eq:masterNullDeriv} in the null diagonal directions, and differentiating $F=\tr A$ gives
\[
\sum_{\alpha\in G}\nabla_\alpha A_{1\alpha}
=F_1+rv_1=(r+\Phi_v)v_1.
\]
Weighted Cauchy--Schwarz yields
\[
\sum_{\alpha\in G}\sum_k\frac{(\nabla_kA_{1\alpha})^2}{a_\alpha}
\ge
\frac{(r+\Phi_v)^2}{F}v_1^2.
\]
Substitution into \eqref{eq:MasterSchurExact}, together with \eqref{eq:MasterDeltaA}, proves \eqref{eq:MasterSchur}.
\end{proof}

\begin{remark}
\label{rem:arbitrarynullity}
The proof involves only a fixed positive-definite principal submatrix at $x_0$.  It does not assume that the rank of $\Hess v$ is locally constant, that its eigenvalues can be smoothly labeled, or that the null eigenvalue is simple.
\end{remark}

\subsection{Inverse-power transforms}
\label{sec:LEstrict}

\begin{proposition}
\label{prop:generalPowerSchur}
Let $u>0$ satisfy $-\Lap u=\lambda u^p$ on an open
set $U\subset\Sn$, where $p>1$, $\lambda>0$, and let
$\beta>0$ be arbitrary. Set
\begin{equation}\label{eq:generalPowerNotation}
v=u^{-\beta},\qquad
c=1+\frac1\beta,\qquad d=\lambda\beta,\qquad
k=1-\frac{p-1}{\beta},\qquad q=|\grad v|^2.
\end{equation}
Then
\begin{equation}\label{eq:generalPowerPDE}
F:=\Lap v=\frac{c q}{v}+d v^k>0.
\end{equation}
Suppose $\Hess v\ge0$ on $U$ and $\Hess v(x_0)$ has a
null direction of rank $r\in\{1,\ldots,n-1\}$.
The nonnegative Schur-complement function $\phi$ in
Lemma~\ref{lem:SchurMaster} satisfies at $x_0$
\begin{equation}\label{eq:generalPowerSchur}
{\begin{aligned}
\Lap\phi\le{}&
\left[\frac{4r-2c}{v}
+d(k+1)(k+2)v^{k-2}
-\frac2F\left(r+d(k+1)v^{k-1}\right)^2\right]v_1^2
-2dv^k.
\end{aligned}}
\end{equation}
\end{proposition}

\begin{proof}
Differentiating $v=u^{-\beta}$ twice and inserting
$-\Lap u=\lambda u^p$ yields \eqref{eq:generalPowerPDE}.
In Lemma~\ref{lem:SchurMaster} take
$\Phi(v,q)=cq/v+dv^k$. Then
\[
\Phi_q=\frac cv,\qquad
\Phi_v=-\frac{cq}{v^2}+dkv^{k-1}
=-\frac Fv+d(k+1)v^{k-1},
\]
and
\[
\Phi_{vv}=\frac{2cq}{v^3}+dk(k-1)v^{k-2}
=\frac{2F}{v^2}+d(k+1)(k-2)v^{k-2}.
\]
The zeroth-order term of \eqref{eq:MasterSchur} is
\[
2\Phi_q q-2F=-2dv^k.
\]
Writing $T=r+d(k+1)v^{k-1}$, the coefficient of $v_1^2$
in that lemma becomes
\[
\frac{2F}{v^2}-\frac{2c}{v}
+d(k+1)(k-2)v^{k-2}
-\frac2F\left(T-\frac Fv\right)^2.
\]
Expanding the last square cancels the terms $2F/v^2$;
the remaining terms are
\[
\frac{4r-2c}{v}
+d(k+1)(k+2)v^{k-2}-\frac{2T^2}{F},
\]
which proves \eqref{eq:generalPowerSchur}.
\end{proof}

\begin{remark}
\label{rem:generalPowerExponent}
When $\beta=(p-1)/3$, one has $k=-2$; hence
$(k+1)(k+2)=0$ and the potentially positive middle term in
\eqref{eq:generalPowerSchur} disappears \emph{identically}.
The remaining rank term is nonpositive whenever
$c\ge2(n-1)$, or equivalently
\[
\beta\le\frac{1}{2n-3}.
\]
For $\beta=(p-1)/3$ this gives precisely
$p\le2n/(2n-3)$. This explains the chosen inverse-power
transform and the endpoint of the present sufficient sign criterion;
it does not establish the optimal exponent for the original PDE.
\end{remark}

\begin{lemma}
\label{lem:Schur}
Let $U\subset\Sn$ be open, $\lambda>0$, and let $u\in C^4(U)$ be positive and satisfy
\begin{equation}
\label{eq:lambdaLE}
-\Lap u=\lambda u^p.
\end{equation}
Assume
\[
1<p\le\frac{2n}{2n-3}.
\]
Define
\[
\beta:=\frac{p-1}{3},\qquad v:=u^{-\beta}.
\]
If $\Hess v\ge0$ in $U$, then
\[
\Hess v>0\qquad\text{in }U.
\]
\end{lemma}

\begin{proof}
Direct differentiation gives
\begin{equation}
\label{eq:vPDEgeneral}
F:=\Lap v
=c\frac qv+d\,v^{-2},
\qquad q:=|\grad v|^2,
\qquad
c:=\frac{p+2}{p-1},
\qquad
d:=\frac{p-1}{3}\lambda>0.
\end{equation}
Thus $F>0$. Apply Proposition~\ref{prop:generalPowerSchur}
with $\beta=(p-1)/3$ and $k=-2$. The middle term
$d(k+1)(k+2)v^{k-2}$ vanishes and
\eqref{eq:generalPowerSchur} reduces to
\begin{equation}
\label{eq:SchurCoreFinal}
\Delta\phi(x_0)
\le
\left[
\frac{4r-2c}{v}
-\frac2F\left(r-dv^{-3}\right)^2
\right]v_1^2
-2dv^{-2}.
\end{equation}
The restriction on $p$ is equivalent to
\[
c=\frac{p+2}{p-1}\ge2(n-1)\ge2r.
\]
Hence every term on the right of \eqref{eq:SchurCoreFinal} is nonpositive except the last, which is strictly negative.  This contradicts \eqref{eq:masterMin}.
\end{proof}

\section{Global strict convexity for the Lane--Emden equation}
\label{sec:LEnear}

\subsection{The near-linear limit}
\label{sec:near1}

Let $\lambda_1$ and $\phi_1>0$ denote the first Dirichlet
eigenvalue and eigenfunction of $-\Lap$ on $\Omega$.
We use the strict spherical log-concavity statement
\begin{equation}\label{eq:LW}
\Hess(\log\phi_1)<0\qquad\text{in }\Omega.
\end{equation}
Spherical log-concavity was established by Lee--Wang
\cite{LeeWang1987}; the strict formulation under the present
smooth strict geodesic convexity assumptions is explicitly stated
in \cite[Proposition~2.1]{GrossiProvenzanoRaom2026}.
The other external tools (whole-space and half-space Liouville
theorems and the Leray--Schauder continuation theorem) are cited
with their relevant hypotheses at the points where they enter
the proofs below. Standard boundary regularity and maximum
principles are used in the forms given in \cite{GilbargTrudinger}.

The first step is to identify the near-linear limit of normalized positive solutions. This leads to uniqueness and nondegeneracy for $p>1$ sufficiently close to $1$.

Normalize the first eigenfunction by
\begin{equation}
\label{eq:L2normphi}
\int_\Omega\phi_1^2\,dV_g=1,
\qquad \phi_1>0.
\end{equation}

\begin{lemma}
\label{lem:maxnorm}
Let $p_j>1$, $p_j\to1$, and let $u_j$ be arbitrary positive solutions of
\[
-\Lap u_j=u_j^{p_j},\qquad u_j|_{\partial\Omega}=0.
\]
Set
\[
M_j:=\|u_j\|_{L^\infty(\Omega)},\qquad
z_j:=\frac{u_j}{M_j},\qquad
\Lambda_j:=M_j^{p_j-1}.
\]
Then
\begin{equation}
\label{eq:LambdaNear1}
\Lambda_j\to\lambda_1,
\end{equation}
and, if $m_1:=\|\phi_1\|_{L^\infty(\Omega)}$,
\begin{equation}
\label{eq:zNear1}
z_j\to\frac{\phi_1}{m_1}
\qquad\text{in }C^{2,\alpha}(\overline\Omega)
\end{equation}
for each fixed $0<\alpha<1$.
\end{lemma}

\begin{proof}
The normalized function satisfies
\begin{equation}
\label{eq:zEq}
-\Lap z_j=\Lambda_j z_j^{p_j},\qquad 0<z_j\le1,\qquad \max_\Omega z_j=1.
\end{equation}
Testing against $\phi_1$ gives
\[
\lambda_1\int_\Omega z_j\phi_1
=\Lambda_j\int_\Omega z_j^{p_j}\phi_1.
\]
Since $0<z_j\le1$ and $p_j>1$, $z_j^{p_j}\le z_j$, hence
\begin{equation}
\label{eq:LambdaLower}
\Lambda_j\ge\lambda_1.
\end{equation}

We first show that $(\Lambda_j)$ is bounded above.  Otherwise, after passing to a subsequence,
\[
\Lambda_j\to+\infty.
\]
Let $x_j$ be a maximum point of $z_j$ and put
\[
\rho_j:=\Lambda_j^{-1/2}\to0,
\qquad d_j:=\dist(x_j,\partial\Omega).
\]
After passing to a subsequence, either $d_j/\rho_j\to+\infty$ or $d_j/\rho_j$ remains bounded.

In the first case, exponential coordinates at $x_j$ and the rescaling $x=\exp_{x_j}(\rho_j y)$ yield metrics converging smoothly on compact sets to the Euclidean metric and functions $Z_j$ satisfying
\[
-\Delta_{g_j}Z_j=Z_j^{p_j},\qquad 0\le Z_j\le1,\qquad Z_j(0)=1.
\]
Interior elliptic estimates and a diagonal argument give a bounded positive limit $Z$ solving
\begin{equation}
\label{eq:linearRnapprox}
-\Delta Z=Z\qquad\text{in }\R^n,
\qquad Z(0)=1.
\end{equation}
This is impossible.  Indeed, choose a Euclidean ball $B_R$ with first Dirichlet eigenvalue $\mu_1(B_R)<1$, and let $\varphi_R>0$ be its first eigenfunction.  Green's identity gives
\[
(1-\mu_1(B_R))\int_{B_R}Z\varphi_R
=\int_{\partial B_R}Z\,\partial_\nu\varphi_R<0,
\]
whereas the left-hand side is positive.

In the second case, use boundary Fermi coordinates centered at a nearest boundary point and rescale by $\rho_j$.  Standard boundary estimates yield a limit on a half-space,
\[
-\Delta Z=Z\quad\text{in }\R^n_+,
\qquad Z=0\quad\text{on }\partial\R^n_+,
\qquad 0\le Z\le1.
\]
Lemma~\ref{lem:quantBoundaryBlowup} shows quantitatively that $d_j/\rho_j\ge c>0$: the rescaled Dirichlet solutions satisfy uniform boundary $C^1$ estimates and have value one at the maximum, so they cannot approach the zero boundary.  Hence the limiting maximum point is interior, and $Z>0$ by the strong maximum principle.  Since $\R^n_+$ contains Euclidean balls of arbitrarily large radius, the preceding first-eigenfunction argument on such a ball again contradicts $-\Delta Z=Z$.  Thus $(\Lambda_j)$ is bounded.

Pass to a subsequence with $\Lambda_j\to\Lambda_*\ge\lambda_1$.
To establish convergence up to the boundary, fix any finite $q>n$.
the common bounds $0\le z_j\le1$ and $\Lambda_j\le C$ in
\eqref{eq:zEq}, together with the global Dirichlet $W^{2,q}$ estimate,
give
\[
\|z_j\|_{W^{2,q}(\Omega)}\le C_q.
\]
For any prescribed $0<\alpha<1$, choose $q$ large enough that
$\gamma:=1-n/q>\alpha$. The Sobolev--Morrey embedding gives a uniform
$C^{1,\gamma}(\overline\Omega)$ bound. For large $j$ we have
$1<p_j\le2$, and $s\mapsto s^{p_j}$ is Lipschitz on $[0,1]$ with
constant at most $2$. Consequently
\[
\|\Lambda_jz_j^{p_j}\|_{C^{0,\gamma}(\overline\Omega)}\le C.
\]
The global Schauder estimate now yields
$\|z_j\|_{C^{2,\gamma}(\overline\Omega)}\le C$; compact embedding gives
convergence in $C^{2,\alpha}(\overline\Omega)$ along a subsequence.
This argument works for any fixed $\alpha<1$. Hence
\[
z_j\to z_*,\qquad -\Lap z_*=\Lambda_*z_*,\qquad z_*\ge0,\qquad \max z_*=1.
\]
The strong maximum principle gives $z_*>0$ in $\Omega$.  A positive Dirichlet eigenfunction can only correspond to $\lambda_1$, so $\Lambda_*=\lambda_1$ and $z_*=\phi_1/m_1$.  Since every convergent subsequence has the same limit, the whole sequence converges.
\end{proof}

\begin{corollary}
\label{cor:intnorm}
For $u_j$ as in Lemma~\ref{lem:maxnorm}, define
\[
a_j:=\int_\Omega u_j\phi_1\,dV_g,
\qquad w_j:=\frac{u_j}{a_j},
\qquad \lambda_j:=a_j^{p_j-1}.
\]
Then
\begin{equation}
\label{eq:wnorm}
\int_\Omega w_j\phi_1=1,
\qquad -\Lap w_j=\lambda_jw_j^{p_j},
\end{equation}
and
\begin{equation}
\label{eq:wlambdaNear1}
w_j\to\phi_1\quad\text{in }C^{2,\alpha}(\overline\Omega),
\qquad \lambda_j\to\lambda_1.
\end{equation}
\end{corollary}

\begin{proof}
By Lemma~\ref{lem:maxnorm},
\[
\frac{a_j}{M_j}
=\int_\Omega z_j\phi_1
\longrightarrow \frac1{m_1}\int_\Omega\phi_1^2
=\frac1{m_1}.
\]
Thus $M_j/a_j\to m_1$, so $w_j=(M_j/a_j)z_j\to\phi_1$.  Moreover
\[
\lambda_j=a_j^{p_j-1}
=\left(\frac{a_j}{M_j}\right)^{p_j-1}\Lambda_j\to\lambda_1.
\]
\end{proof}

\subsection{Near-linear uniqueness and nondegeneracy}

\begin{lemma}
\label{lem:existence}
For every $1<p<p_S(n)$, problem \eqref{eq:mainLE} has at least one positive solution.
\end{lemma}

\begin{proof}
The embedding $H_0^1(\Omega)\hookrightarrow L^{p+1}(\Omega)$ is compact.  Minimize
\[
\frac{\int_\Omega|\grad v|^2}{\left(\int_\Omega|v|^{p+1}\right)^{2/(p+1)}}
\]
over nonzero $v\in H_0^1(\Omega)$.  A minimizer may be chosen nonnegative and satisfies
\[
-\Lap v=\kappa v^p
\]
for some $\kappa>0$.  The strong maximum principle gives $v>0$, and scaling $u=\kappa^{1/(p-1)}v$ gives \eqref{eq:mainLE}.  Standard elliptic regularity yields a classical solution.
\end{proof}

\begin{proposition}
\label{prop:near1unique}
There exists $\varepsilon_0>0$ such that for every
\[
1<p<1+\varepsilon_0
\]
problem \eqref{eq:mainLE} has exactly one positive solution.
\end{proposition}

\begin{proof}
Existence follows from Lemma~\ref{lem:existence} after shrinking $\varepsilon_0$ so that $1+\varepsilon_0<p_S(n)$ when $n\ge3$.

Suppose uniqueness fails.  Then there exist $p_j\downarrow1$ and two distinct positive solutions $u_j,\widetilde u_j$.  Normalize them by
\[
a_j=\int u_j\phi_1,\quad w_j=u_j/a_j,\quad \lambda_j=a_j^{p_j-1},
\]
and similarly $\widetilde a_j,\widetilde w_j,\widetilde\lambda_j$.  Corollary~\ref{cor:intnorm} yields
\[
w_j,\widetilde w_j\to\phi_1\quad\text{in }C^{2,\alpha},
\qquad \lambda_j,\widetilde\lambda_j\to\lambda_1.
\]
If $w_j=\widetilde w_j$, then the two equations imply $\lambda_j=\widetilde\lambda_j$, hence $a_j=\widetilde a_j$ because $p_j>1$, contradicting $u_j\ne\widetilde u_j$.  Thus
\[
d_j:=\|w_j-\widetilde w_j\|_\infty>0,
\qquad \eta_j:=\frac{w_j-\widetilde w_j}{d_j}.
\]
Then
\begin{equation}
\label{eq:etaOrth}
\|\eta_j\|_\infty=1,
\qquad \int_\Omega\eta_j\phi_1=0.
\end{equation}
Testing the normalized equations against $\phi_1$ gives
\[
\lambda_j=\frac{\lambda_1}{\int w_j^{p_j}\phi_1},
\qquad
\widetilde\lambda_j=\frac{\lambda_1}{\int \widetilde w_j^{p_j}\phi_1}.
\]
The denominators converge to $1$ and, on the common bounded range of $w_j$ and $\widetilde w_j$, the maps $s\mapsto s^{p_j}$ are uniformly Lipschitz.  Consequently
\begin{equation}
\label{eq:lambdaLip}
|\lambda_j-\widetilde\lambda_j|\le C d_j.
\end{equation}
Write
\[
A_j(x):=p_j\int_0^1\bigl(tw_j+(1-t)\widetilde w_j\bigr)^{p_j-1}\,dt,
\qquad
\theta_j:=\frac{\lambda_j-\widetilde\lambda_j}{d_j}.
\]
Then $(\theta_j)$ is bounded and
\begin{equation}
\label{eq:etaeq}
-\Lap\eta_j=\lambda_jA_j\eta_j+\theta_j\widetilde w_j^{p_j}.
\end{equation}
The functions $A_j$ are uniformly bounded and converge pointwise in $\Omega$ to $1$; hence $A_j\to1$ in $L^q(\Omega)$ for every finite $q$.  Elliptic estimates applied to \eqref{eq:etaeq} give, after passing to a subsequence,
\[
\eta_j\to\eta\quad\text{uniformly},\qquad \theta_j\to\theta,
\]
and the limit solves
\begin{equation}
\label{eq:etalimit}
-\Lap\eta=\lambda_1\eta+\theta\phi_1,
\qquad \int_\Omega\eta\phi_1=0.
\end{equation}
Testing \eqref{eq:etalimit} against $\phi_1$ gives $\theta=0$.  Therefore $\eta$ is a $\lambda_1$-eigenfunction orthogonal to $\phi_1$, hence $\eta=0$.  This contradicts $\|\eta\|_\infty=1$, which follows from uniform convergence and \eqref{eq:etaOrth}.
\end{proof}

\begin{proposition}
\label{prop:near1nondeg}
After possibly decreasing $\varepsilon_0$, the unique positive solution $u_p$ is nondegenerate for every $1<p<1+\varepsilon_0$:
\[
\Ker\bigl(-\Lap-pu_p^{p-1}\bigr)=\{0\}
\quad\text{in }H_0^1(\Omega).
\]
\end{proposition}

\begin{proof}
Suppose not.  Then there exist $p_j\downarrow1$, positive solutions $u_j$, and nonzero $\psi_j\in H_0^1(\Omega)$ such that
\[
-\Lap\psi_j=p_ju_j^{p_j-1}\psi_j.
\]
Normalize $\|\psi_j\|_{L^2}=1$, and write $u_j=a_jw_j$ with the integral normalization of Corollary~\ref{cor:intnorm}.  Then
\[
p_ju_j^{p_j-1}=p_j\lambda_jw_j^{p_j-1}=:V_j.
\]
Since $w_j\to\phi_1$ uniformly, $(w_j)$ is uniformly bounded, and $p_j-1\to0$, the potentials $V_j$ are uniformly bounded and
\begin{equation}
\label{eq:Vconv}
V_j\to\lambda_1\qquad\text{in }L^q(\Omega)
\end{equation}
for every finite $q$, by dominated convergence.  Elliptic estimates give compactness of $(\psi_j)$ in $L^2$; passing to a subsequence,
\[
\psi_j\to\psi\quad\text{strongly in }L^2,
\]
and the limit equation is
\[
-\Lap\psi=\lambda_1\psi,
\qquad \|\psi\|_{L^2}=1.
\]
Thus $\psi=c\phi_1$ with $c\ne0$.

Self-adjointness yields the following orthogonality identity.  Since
\[
(-\Lap-p_ju_j^{p_j-1})u_j=(1-p_j)u_j^{p_j},
\]
we have
\[
0=\int_\Omega u_j(-\Lap-p_ju_j^{p_j-1})\psi_j
=(1-p_j)\int_\Omega u_j^{p_j}\psi_j.
\]
Therefore
\[
\int_\Omega w_j^{p_j}\psi_j=0.
\]
Because $w_j^{p_j}\to\phi_1$ uniformly, passing to the limit yields
\[
\int_\Omega\phi_1\psi=0,
\]
contradicting $\psi=c\phi_1\ne0$.
\end{proof}

\subsection{Boundary convexity and the initial solution curve}
\label{sec:collar}

\begin{lemma}
\label{lem:collar}
Let $\mathscr K$ be a compact family in $C^2(\overline\Omega)$ of positive functions $u$ satisfying $u=0$ on $\partial\Omega$ and $-\Lap u>0$ in $\Omega$.  Let $\beta$ range in a compact subset of $[0,\infty)$, with the conclusion considered for $\beta>0$.  Then there exists $\delta>0$, uniform over the family, such that
\begin{equation}
\label{eq:collarResult}
\Hess(u^{-\beta})>0
\qquad\text{whenever }0<\dist(x,\partial\Omega)<\delta.
\end{equation}
\end{lemma}

\begin{proof}
By Hopf's lemma, $\partial_\nu u<0$ on $\partial\Omega$ for every member of the family.  Compactness of $\mathscr K$ in $C^1$ and compactness of $\partial\Omega$ imply the uniform bound
\begin{equation}
\label{eq:uniformHopf}
-\partial_\nu u\ge2a_0>0
\qquad\text{on }\partial\Omega.
\end{equation}
Indeed, otherwise a convergent sequence in $\mathscr K$ would have a limit violating Hopf's lemma.  Shrinking a tubular neighborhood uniformly and using $C^1$ compactness gives
\begin{equation}
\label{eq:gradlower}
|\grad u|\ge a_0
\end{equation}
there.  Also
\begin{equation}
\label{eq:Hessupper}
|\Hess u|\le C_0
\end{equation}
uniformly.

For $\tau\in T_x\partial\Omega$, differentiating $u|_{\partial\Omega}=0$ twice tangentially gives
\begin{equation}
\label{eq:boundaryHessIdentity}
\Hess u(\tau,\tau)=\partial_\nu u\,\II(\tau,\tau)
\le-2a_0\kappa_0|\tau|^2.
\end{equation}
We next show that, after shrinking the tubular neighborhood, there exists $b_0>0$ such that
\begin{equation}
\label{eq:levelTangential}
\Hess u(\tau,\tau)\le-b_0|\tau|^2
\end{equation}
for every vector $\tau$ satisfying $\tau\perp\grad u$.  If not, there would exist $u_j\in\mathscr K$, points $x_j\to x_*\in\partial\Omega$, and unit vectors $\tau_j\perp\grad u_j(x_j)$ for which the left side of \eqref{eq:levelTangential} tends to a nonnegative value.  Passing to a $C^2$ limit and parallel-transporting the vectors in a local chart gives $\tau_*\perp\grad u_*(x_*)$.  Since $\grad u_*(x_*)$ is a nonzero normal vector by Hopf's lemma, $\tau_*\in T_{x_*}\partial\Omega$, contradicting \eqref{eq:boundaryHessIdentity}.  We may take, after a further shrinking, $b_0=a_0\kappa_0$.

Fix $\beta>0$ and set
\[
v=u^{-\beta}.
\]
A direct differentiation gives
\begin{equation}
\label{eq:powerHess}
\Hess v
=\beta u^{-\beta-2}\bigl[(\beta+1)\,du\otimes du-u\Hess u\bigr].
\end{equation}
Let
\[
N:=\frac{\grad u}{|\grad u|},\qquad X=\tau+sN,\qquad \tau\perp\grad u.
\]
Using \eqref{eq:gradlower}--\eqref{eq:levelTangential},
\[
\begin{aligned}
&\bigl[(\beta+1)du\otimes du-u\Hess u\bigr](X,X)\\
&\quad\ge (\beta+1)a_0^2s^2+b_0u|\tau|^2-2C_0u|s||\tau|-C_0us^2.
\end{aligned}
\]
Young's inequality gives
\[
2C_0u|s||\tau|
\le\frac{b_0}{2}u|\tau|^2+\frac{2C_0^2}{b_0}us^2,
\]
so
\begin{equation}
\label{eq:collarSquare}
\begin{aligned}
&\bigl[(\beta+1)du\otimes du-u\Hess u\bigr](X,X)\\
&\quad\ge \frac{b_0}{2}u|\tau|^2
+\left[(\beta+1)a_0^2-\left(C_0+\frac{2C_0^2}{b_0}\right)u\right]s^2.
\end{aligned}
\end{equation}
The family is uniformly $C^1$ and vanishes on $\partial\Omega$, hence
\[
u(x)\le C_1\dist(x,\partial\Omega).
\]
Since $\beta+1\ge1$, choose $\delta>0$ so that
\[
\left(C_0+\frac{2C_0^2}{b_0}\right)C_1\delta\le\frac{a_0^2}{2}.
\]
Then the right-hand side of \eqref{eq:collarSquare} is positive for every nonzero $X$ in the collar.  Equation \eqref{eq:powerHess} proves \eqref{eq:collarResult}.
\end{proof}

\begin{proposition}
\label{prop:near1convex}
There exists $\varepsilon_1>0$ such that the unique positive solution $u_p$ satisfies
\[
\Hess\!\left(u_p^{-\frac{p-1}{3}}\right)>0
\qquad\text{in }\Omega
\]
for every $1<p<1+\varepsilon_1$.
\end{proposition}

\begin{proof}
Use the integral normalization $u_p=a_pw_p$.  By Corollary~\ref{cor:intnorm}, $w_p\to\phi_1$ in $C^{2,\alpha}(\overline\Omega)$ as $p\downarrow1$.  The family $\{w_p\}$ together with its limit $\phi_1$ is therefore compact in $C^2$.  Lemma~\ref{lem:collar}, applied with $\beta_p=(p-1)/3$, gives a uniform boundary collar on which $\Hess(w_p^{-\beta_p})>0$.

On the compact subset $K\Subset\Omega$ remaining after removal of this boundary neighborhood, Lee--Wang's theorem \eqref{eq:LW} and compactness give
\[
-\Hess(\log\phi_1)\ge2\sigma_Kg
\]
for some $\sigma_K>0$.  Since $w_p\to\phi_1$ in $C^2(K)$ and $\phi_1$ is bounded away from zero on $K$,
\[
-\Hess(\log w_p)\ge\sigma_Kg
\]
for $p$ sufficiently close to one.  Finally,
\begin{equation}
\label{eq:logtopower}
\Hess(w_p^{-\beta_p})
=\beta_pw_p^{-\beta_p}
\bigl[-\Hess(\log w_p)+\beta_p\,d\log w_p\otimes d\log w_p\bigr]>0.
\end{equation}
Multiplication by the positive constant $a_p^{-\beta_p}$ shows the same for $u_p^{-\beta_p}$.
\end{proof}

\subsection{A priori bounds on parameter strips}
\label{sec:Linfty}

\begin{proposition}
\label{prop:Linfty}
Let
\[
1<p_-\le p\le p_+<p_S(n).
\]
Then there exists
\[
C=C(\Omega,n,p_-,p_+)>0
\]
such that every positive solution of \eqref{eq:mainLE} satisfies
\[
\|u\|_{L^\infty(\Omega)}\le C.
\]
No convexity assumption is needed for this proposition.
\end{proposition}

\begin{proof}
Suppose there are $p_j\in[p_-,p_+]$ and positive solutions $u_j$ with
\[
M_j:=\|u_j\|_\infty\to+\infty.
\]
Let $x_j$ be a maximum point and set
\begin{equation}
\label{eq:blowscale}
\rho_j:=M_j^{-(p_j-1)/2}\to0.
\end{equation}
After passing to a subsequence, $p_j\to p_\infty\in[p_-,p_+]$.  Put $d_j=\dist(x_j,\partial\Omega)$.

\medskip
\noindent\emph{Interior blow-up.}
Assume first $d_j/\rho_j\to+\infty$.  In exponential coordinates at $x_j$ define
\[
U_j(y):=M_j^{-1}u_j(\exp_{x_j}(\rho_jy)),
\qquad g_j:=\rho_j^{-2}(\exp_{x_j}\circ(\rho_j\,\cdot))^*g.
\]
Then on every fixed ball,
\[
g_j\to g_{\rm Euc}\quad\text{smoothly},
\qquad -\Delta_{g_j}U_j=U_j^{p_j},
\qquad 0<U_j\le1,
\qquad U_j(0)=1.
\]
Interior estimates yield
\[
U_j\to U\quad\text{in }C^2_{\rm loc}(\R^n),
\]
where
\[
-\Delta U=U^{p_\infty}\quad\text{in }\R^n,
\qquad U(0)=1,
\qquad 0<U\le1.
\]
The Gidas--Spruck whole-space Liouville theorem
\cite{GidasSpruck1981} excludes nonzero nonnegative classical
solutions of $-\Delta U=U^q$ on $\R^n$ for
$1<q<(n+2)/(n-2)$ when $n\ge3$, and for every $q>1$
when $n=2$. Since $p_\infty<p_S(n)$, this is a contradiction.

\medskip
\noindent\emph{Boundary blow-up.}
Assume now $d_j/\rho_j$ is bounded.  Then $d_j\to0$.  Choose a nearest boundary point $z_j$ and use uniform Fermi coordinates
\[
\Psi_j(y',s)=\exp_{\Theta_j(y')}(-s\nu(\Theta_j(y')))
\]
centered at $z_j$.  Rescale by $\rho_j$ and define
\[
U_j(y',y_n)
:=M_j^{-1}u_j(\Psi_j(\rho_jy',\rho_jy_n)).
\]
The rescaled metrics converge smoothly on compact half-balls to the Euclidean metric; the boundary becomes $\{y_n=0\}$, and
\[
-\Delta_{g_j}U_j=U_j^{p_j},
\qquad U_j=0\text{ on }\{y_n=0\},
\qquad 0\le U_j\le1.
\]
The rescaled maximum point is $(0,d_j/\rho_j)$.  After passing to a subsequence,
\[
\frac{d_j}{\rho_j}\to a\in[0,\infty).
\]
The quantitative boundary estimates of Lemma~\ref{lem:quantBoundaryBlowup} give $d_j/\rho_j\ge c>0$ and local $C^{2,\alpha}$ convergence up to the flat boundary to a function $U$ satisfying
\[
-\Delta U=U^{p_\infty}\quad\text{in }\R^n_+,
\qquad U=0\quad\text{on }\partial\R^n_+,
\qquad 0\le U\le1.
\]
Lemma~\ref{lem:quantBoundaryBlowup} gives $a\ge c>0$ and $U(ae_n)=1$, so $U$ is a bounded positive classical half-space solution.  By Dupaigne--Sirakov--Souplet
\cite[Corollary~1.2]{DupaigneSirakovSouplet2022},
the half-space Dirichlet equation $-\Delta U=U^q$, $q>1$,
admits no nonzero nonnegative classical solution bounded on
every finite-height strip. Here $q=p_\infty>1$, and the limit
is globally bounded by one, a contradiction.

Both alternatives lead to contradictions, and the desired uniform bound follows.
\end{proof}

\begin{remark}
For our final range,
\[
p_n^*=\frac{2n}{2n-3}<\frac{n+2}{n-2}=p_S(n)
\qquad(n\ge3),
\]
because
\[
(n+2)(2n-3)-2n(n-2)=5n-6>0.
\]
For $n=2$, $p_S(2)=+\infty$.  Hence Proposition~\ref{prop:Linfty} applies on every compact interval $[p_0,p_n^*]$ with $p_0>1$.
\end{remark}

\subsection{Global continuation and strict convexity}
\label{sec:continuum}

In this subsection, let $p_0>1$ be any exponent sufficiently
close to $1$ that Propositions~\ref{prop:near1unique},
\ref{prop:near1nondeg}, and \ref{prop:near1convex} apply, and write
\[
P:=p_n^*.
\]
The argument is uniform with respect to any such fixed choice of $p_0$.

\begin{proposition}
\label{prop:continuum}
There exists a compact connected set
\[
\mathscr C\subset[p_0,P]\times C_0(\overline\Omega)
\]
consisting of positive solutions of \eqref{eq:mainLE}, containing the unique point $(p_0,u_{p_0})$ on the left parameter face, and satisfying
\[
\pi_p(\mathscr C)=[p_0,P].
\]
Moreover, $\mathscr C$ is compact in $[p_0,P]\times C^{2,\alpha'}(\overline\Omega)$ for every sufficiently small $\alpha'>0$.
\end{proposition}

\begin{proof}
Let
\[
X:=C_0(\overline\Omega)
\]
with the sup norm, and let $G=(-\Lap)^{-1}$ denote the Dirichlet Green operator.  Define
\begin{equation}
\label{eq:Tmap}
T(p,u):=G((u_+)^p),
\qquad (p,u)\in[p_0,P]\times X.
\end{equation}
The map $(p,s)\mapsto(s_+)^p$ is uniformly continuous on compact subsets of $[p_0,P]\times\R$.  Thus $T$ is continuous.  If $B\subset[p_0,P]\times X$ is bounded, then $(u_+)^p$ is uniformly bounded in $L^q$ for every finite $q$.  Choosing $q>n$, global $W^{2,q}$ estimates give
\[
\|T(p,u)\|_{W^{2,q}}\le C_B,
\]
and $W^{2,q}\Subset C^0$.  Hence $T$ is completely continuous.

If $u=T(p,u)$, then
\[
-\Lap u=(u_+)^p\ge0,
\qquad u|_{\partial\Omega}=0.
\]
The maximum principle gives $u\ge0$, so either $u\equiv0$ or $u>0$ and $u$ is a positive solution of \eqref{eq:mainLE}.  Consequently, nonzero fixed points correspond precisely to positive solutions.

By Proposition~\ref{prop:Linfty}, there exists $M_*>0$ such that every positive solution for $p\in[p_0,P]$ satisfies $\|u\|_\infty\le M_*$.  There is also a uniform positive lower bound.  Indeed, testing against $\phi_1$ yields
\[
\lambda_1\int u\phi_1=\int u^p\phi_1
\le\|u\|_\infty^{p-1}\int u\phi_1,
\]
so
\[
\|u\|_\infty\ge\lambda_1^{1/(p-1)}.
\]
Therefore
\[
m_*:=\min_{p\in[p_0,P]}\lambda_1^{1/(p-1)}>0.
\]
Choose
\[
0<r<m_*,\qquad R>M_*,
\qquad \mathcal O:=\{u\in X:r<\|u\|_\infty<R\}.
\]
There is no fixed point of $T(p,\cdot)$ on $\partial\mathcal O$
for any $p\in[p_0,P]$. The set $\mathcal O$ is a bounded open annulus in $X$.
No convexity of $\mathcal O$ is required in the continuation theorem.

At $p=p_0$, Proposition~\ref{prop:near1unique} gives a unique nonzero fixed point $u_{p_0}$ in $\mathcal O$.  Since $p_0>1$, the Nemytskii map $u\mapsto(u_+)^{p_0}$ is $C^1$ on $C^0$ bounded sets, and
\[
D_uT(p_0,u_{p_0})h
=G(p_0u_{p_0}^{p_0-1}h).
\]
If $(I-D_uT)h=0$, then
\[
(-\Lap-p_0u_{p_0}^{p_0-1})h=0.
\]
Proposition~\ref{prop:near1nondeg} therefore gives $h=0$.  Because $D_uT$ is compact, $I-D_uT$ is Fredholm of index zero and hence invertible.  The local Leray--Schauder degree at $u_{p_0}$ is consequently $\pm1$.  By uniqueness and excision,
\[
\deg_{LS}(I-T(p_0,\cdot),\mathcal O,0)=\pm1\ne0.
\]
We apply Mawhin's Leray--Schauder continuation theorem
\cite[Theorem~1]{Mawhin1997} to the bounded open cylinder
$\mathcal O\times[p_0,P]$ (with an affine reparametrization of
$[p_0,P]$ to $[0,1]$). The assumptions in the original formulation
are: complete continuity, absence of fixed points on the lateral
boundary, a finite nonempty set of fixed points on the initial
face, and nonzero Leray--Schauder degree on that face.
They hold here because $T$ is completely continuous,
$T(p,\cdot)$ has no fixed point on $\partial\mathcal O$ for any
$p\in[p_0,P]$, the initial fixed-point set is
$\{u_{p_0}\}$, and the degree equals $\pm1$.
Mawhin also notes that the finiteness requirement can be omitted,
but it is satisfied in the present application.
Thus there exists a compact connected set $\mathscr C$ of fixed
points meeting both parameter faces.
Its parameter projection is connected and contains $p_0$ and $P$, hence equals $[p_0,P]$.

Finally, the uniform $L^\infty$ bound and the equation imply uniform global $W^{2,q}$ estimates.  After $q>n$ gives uniform $C^{1,\gamma}$ control, the nonlinearities $u^p$ have a uniform $C^{0,\gamma}$ bound because $p\in[p_0,P]$ and $s\mapsto s^p$ is uniformly Lipschitz on the common bounded range.  Schauder estimates then give a uniform $C^{2,\gamma}$ bound.  Consequently $C^0$ convergence inside $\mathscr C$ automatically upgrades, after subsequences and uniqueness of the $C^0$ limit, to $C^{2,\alpha'}$ convergence for every $0<\alpha'<\gamma$.
\end{proof}

\begin{proposition}\label{prop:openclosed}
Let $\mathscr S$ be the full positive-solution set defined in
Theorem~\ref{thm:main}, with the topology induced by
$(1,P]\times C_0(\overline\Omega)$. Then
\[
\mathscr G:=\bigl\{(p,u)\in\mathscr S:
\Hess(u^{-(p-1)/3})>0\text{ in }\Omega\bigr\}
\]
is both relatively open and relatively closed in $\mathscr S$.
Consequently, every connected component of $\mathscr S$ is either
contained in $\mathscr G$ or disjoint from $\mathscr G$.
\end{proposition}

\begin{proof}
We first record the topology needed for this assertion. For each
compact interval $I=[a,b]\subset(1,P]$, the set of all positive
solutions with $p\in I$ is compact in $I\times C_0(\overline\Omega)$.
Indeed, Proposition~\ref{prop:Linfty} gives a uniform upper bound.
Testing against the positive first Dirichlet eigenfunction gives
\[
\|u\|_\infty^{p-1}\ge\lambda_1,
\]
so the solutions have a uniform positive lower bound in the
sup norm. Global elliptic estimates and a bootstrap through
$W^{2,q}$ and $C^{2,\gamma}$, for a sufficiently small fixed
$\gamma>0$, then give precompactness in $C^2(\overline\Omega)$.
Every limit satisfies the equation and is nontrivial, hence strictly
positive by the strong maximum principle. In particular, if
$(p_j,u_j)\to(p,u)$ in $\mathscr S$ with the indicated product
topology, then
\begin{equation}\label{eq:componentCtwo}
u_j\to u\qquad\text{in }C^2(\overline\Omega).
\end{equation}

\emph{Relative openness.}
Take $(p,u)\in\mathscr G$ and a sequence
$(p_j,u_j)\in\mathscr S$ converging to $(p,u)$. We show that
$(p_j,u_j)\in\mathscr G$ for all sufficiently large $j$.
By \eqref{eq:componentCtwo}, the family consisting of $u$ and the
functions $u_j$ (after discarding finitely many terms) is compact
in $C^2(\overline\Omega)$ and consists of positive solutions.
The exponents $\beta_j=(p_j-1)/3$ lie in a compact subinterval
of $(0,\infty)$. Lemma~\ref{lem:collar} therefore supplies a
common boundary collar on which all the transformed Hessians
are positive definite. On the complementary compact subset
$K\Subset\Omega$, the Hessian of $u^{-(p-1)/3}$ has a strictly
positive lower eigenvalue, and $u$ is bounded away from zero.
The convergence \eqref{eq:componentCtwo} and $p_j\to p$ give
uniform convergence of the transformed Hessians on $K$.
They are therefore positive definite on $K$ for all sufficiently
large $j$, proving relative openness.

\emph{Relative closedness.}
Suppose $(p_j,u_j)\in\mathscr G$ and
$(p_j,u_j)\to(p,u)\in\mathscr S$. By \eqref{eq:componentCtwo},
on each compact subset of $\Omega$ we have
\[
\Hess(u_j^{-(p_j-1)/3})
 \longrightarrow\Hess(u^{-(p-1)/3})
\quad\text{uniformly}.
\]
Thus $\Hess(u^{-(p-1)/3})\ge0$ throughout $\Omega$.
Since $p\le P$, Lemma~\ref{lem:Schur} gives the strict inequality,
and $(p,u)\in\mathscr G$. This proves relative closedness.

Finally, the intersection of $\mathscr G$ with a connected
component of $\mathscr S$ is both relatively open and relatively
closed in that component, so it is either empty or the entire
component.
\end{proof}

\begin{proof}[Proof of Theorem~\ref{thm:main}]
Choose $\varepsilon_\Omega>0$ so small that
$1+\varepsilon_\Omega<P$ and the conclusions of
Propositions~\ref{prop:near1unique}, \ref{prop:near1nondeg},
and \ref{prop:near1convex} all hold for
$1<p<1+\varepsilon_\Omega$.
For these exponents let $u_p$ denote the unique positive solution.
The map $p\mapsto u_p$ is continuous in $C_0(\overline\Omega)$:
if $p_j\to p\in(1,1+\varepsilon_\Omega)$, the compactness
argument in Proposition~\ref{prop:openclosed} yields a convergent
subsequence of $u_{p_j}$, and its limit must equal $u_p$ by
uniqueness. Hence $\Gamma$ is connected and is contained in
$\mathscr G$.

Fix $p_c\in(1,1+\varepsilon_\Omega)$. For each
$p_0\in(1,p_c)$, Proposition~\ref{prop:continuum} furnishes a
compact connected set $\mathscr C_{p_0}\subset\mathscr S$ with
\[
\pi_p(\mathscr C_{p_0})=[p_0,P].
\]
Because $p_c\in[p_0,P]$ and the solution at $p_c$ is unique,
\[
(p_c,u_{p_c})\in\mathscr C_{p_0}\cap\Gamma.
\]
Consequently,
\[
\mathscr A:=\Gamma\cup
\bigcup_{1<p_0<p_c}\mathscr C_{p_0}
\]
is connected. Its parameter projection is $(1,P]$: the
near-linear curve covers parameters arbitrarily close to $1$,
and each $\mathscr C_{p_0}$ reaches $P$. Therefore the connected
component $\mathscr K$ of $\mathscr S$ containing $\mathscr A$
satisfies
\[
\pi_p(\mathscr K)=(1,P].
\]
Since $\mathscr G$ is relatively open and closed in $\mathscr S$
by Proposition~\ref{prop:openclosed}, and $\Gamma\subset\mathscr G$,
connectedness implies
\[
\mathscr K\subset\mathscr G.
\]
This proves \eqref{eq:mainconv} for every point of the component,
not merely for points of a selected Leray--Schauder continuum.
The component $\mathscr K$ is closed in $\mathscr S$, and the
compactness on every parameter strip $[a,P]$ established in the
proof of Proposition~\ref{prop:openclosed} gives the asserted
compactness of $\mathscr K\cap([a,P]\times X)$.

Finally, fix any $(p,u)\in\mathscr K$ and write
\[
\beta=(p-1)/3,\qquad v=u^{-\beta}.
\]
As $u\to0$ at the boundary, $v\to+\infty$ there and attains
an interior minimum. Since
\[
dv=-\beta u^{-\beta-1}du,
\]
$u$ and $v$ have exactly the same critical points. The strict
convexity $\Hess v>0$ along geodesics in $\Omega$ implies that
$v$ has only one critical point, necessarily its minimum; thus
$u$ has a unique maximum. At this point,
\[
\Hess v=-\beta u^{-\beta-1}\Hess u,
\]
so $\Hess u<0$ and the maximum is nondegenerate.

For $0<t<\max_\Omega u$ we have
\[
\{u>t\}=\{v<t^{-\beta}\}.
\]
These sets are geodesically convex by the strict convexity of $v$.
Every such level is regular because the sole critical point of
$u$ is its maximum. If $X$ is tangent to a level of $v$ and
$\nu=\grad v/|\grad v|$ is the outward normal of its sublevel
set, then
\[
\II(X,X)=\frac{\Hess v(X,X)}{|\grad v|}>0.
\]
Thus all positive superlevel sets are strictly geodesically convex.
\end{proof}

\section{Strict concavity for spherical torsion}
\label{sec:torsionglobal}

For the torsion problem we set $v=-\sqrt u$. The proof first excludes null directions of a nonnegative Hessian and then propagates strict convexity along a deformation from a geodesic ball to $\Omega$.

\subsection{Local strictification}

\begin{lemma}
\label{lem:torsionStrict}
Let $U\subset\Sn$ be open and let $v\in C^4(U)$ satisfy
\begin{equation}
\label{eq:torsionTransformed}
v<0,
\qquad
\Lap v=-\frac{|\grad v|^2+\frac12}{v}.
\end{equation}
If $\Hess v\ge0$ in $U$, then $\Hess v>0$ in $U$.
\end{lemma}

\begin{proof}
Write $q=|\grad v|^2$ and
\[
\Phi(v,q)=-\frac{q+\frac12}{v}.
\]
If the Hessian had a null direction at a point where its rank is $r$, Lemma~\ref{lem:SchurMaster} gives
\[
\Delta\phi(x_0)
\le
\left[
\frac{2F}{v^2}+\frac2v
-\frac2F\left(r-\frac Fv\right)^2
\right]v_1^2+\frac1v.
\]
Expanding the square and using $F=-(q+\frac12)/v$ gives
\begin{equation}
\label{eq:torsionSchurFinal}
\Delta\phi(x_0)
\le
\left(
\frac{4r+2}{v}-\frac{2r^2}{F}
\right)v_1^2+\frac1v<0,
\end{equation}
 because $v<0$ and $F>0$. This contradicts \eqref{eq:masterMin}.
\end{proof}

Let $u_\Omega$ denote the torsion function of a domain $\Omega$ and set
\[
v_\Omega=-\sqrt{u_\Omega}.
\]
Then
\[
\Lap v_\Omega=-\frac{|\grad v_\Omega|^2+\frac12}{v_\Omega}>0,
\]
and
\begin{equation}
\label{eq:torsionChain}
4u_\Omega^{3/2}\Hess(-\sqrt{u_\Omega})
=du_\Omega\otimes du_\Omega-2u_\Omega\Hess u_\Omega.
\end{equation}
The transformed function is smooth in the interior but need not be $C^2$ up to the boundary, so the boundary estimates are formulated for $u_\Omega$ and transferred through \eqref{eq:torsionChain}.

\subsection{Geometry of the domain deformation}

\begin{lemma}
\label{lem:hemisphere}
If $\Omega\subsetneq\Sn$ is smooth and uniformly geodesically convex, then $\overline\Omega$ is compactly contained in an open hemisphere.
\end{lemma}

\begin{proof}
Put $K=\overline\Omega$ and consider the cone
\[
C=\{tx:t\ge0,\ x\in K\}\subset\mathbb R^{n+1}.
\]
To check that $C$ is convex, let $a,b\ge0$ and $x,y\in K$.
If $ax+by\ne0$ and $x,y$ are nonantipodal, the normalization
$(ax+by)/|ax+by|$ belongs to the shorter great-circle segment from
$x$ to $y$; hence it belongs to $K$. For antipodal points the sum
lies on one of the two given rays (or vanishes). It follows that
$ax+by\in C$ in every case. Moreover, $C$ is closed because $K$
is compact and $|tx|=t$, and it has nonempty interior since
$\Omega$ is open.

We claim that $C$ contains no line. Indeed, if $\mathbb R\ell\subset C$
for some $\ell\ne0$, convexity and the cone property imply
$C+\mathbb R\ell=C$: every $z\in C$ and every $s\in\mathbb R$
satisfy $z+s\ell\in C$, and the reverse inclusion is immediate.
For a unit vector $x\in K$, the conical structure yields
\[
x\in\partial K\quad\Longleftrightarrow\quad x\in\partial C.
\]
Here $\partial C$ is the Euclidean boundary, while $\partial K$
is the boundary relative to the sphere. To see this equivalence,
intersect a conical neighborhood of the ray $\{rx:r>0\}$ with
$\Sn$; radial normalization is locally a product chart near $x$.
Choose $x\in\partial K$ not parallel to $\ell$; such a choice exists
because $n\ge2$ and the smooth hypersurface $\partial\Omega$
cannot be contained in the two-point set
$\{\ell/|\ell|,-\ell/|\ell|\}$. Translation by $s\ell$
preserves both $C$ and $\partial C$, so for small $s$
\[
\gamma(s)=\frac{x+s\ell}{|x+s\ell|}\in\partial K.
\]
Its image is a nonconstant arc of the great circle
$\Sn\cap\operatorname{span}\{x,\ell\}$. Reparameterize by arc length.
The resulting curve is a unit-speed geodesic contained in
$\partial\Omega$, and thus
\[
\II(\dot\gamma,\dot\gamma)
=-\langle\nabla_{\dot\gamma}\dot\gamma,\nu\rangle=0,
\]
contrary to the positive definiteness of $\II$. Hence $C$ is a
\emph{pointed} closed convex cone.

Consequently $0\notin\operatorname{conv}K$. If otherwise
$0=\sum_{j=1}^N b_jx_j$ with $x_j\in K$, $b_j>0$, and
$\sum_jb_j=1$, then
\[
-x_1=\sum_{j=2}^N\frac{b_j}{b_1}x_j\in C,
\qquad x_1\in C,
\]
so $C$ would contain the whole line $\mathbb Rx_1$.
The convex hull $\operatorname{conv}K$ is compact (in finite
dimension). Let $z$ be its point of minimum Euclidean norm.
Then $z\ne0$, and the first variation of
$|z+s(x-z)|^2$ at $s=0^+$ gives
\[
z\cdot x\ge |z|^2\qquad\text{for every }x\in K.
\]
Thus, with $p=z/|z|$ and $\eta=|z|>0$,
\[
p\cdot x\ge\eta>0\qquad(x\in\overline\Omega),
\]
which proves that the closure lies compactly in the open hemisphere
centered at $p$.
\end{proof}

\begin{lemma}
\label{lem:domainPath}
There exists a smooth family of smooth geodesically convex domains $\Omega_t$, $0\le t\le1$, with $\Omega_1=\Omega$ and $\Omega_0$ a geodesic ball of radius less than $\pi/2$.  The entire family is compactly contained in one open hemisphere and has a parameter-independent positive lower bound for the outward second fundamental form.
\end{lemma}

\begin{proof}
Rotate the hemisphere from Lemma~\ref{lem:hemisphere} to
$x_{n+1}>0$ and use central projection
\[
G(x)=\frac{(x_1,\dots,x_n)}{x_{n+1}},
\qquad
P(y)=G^{-1}(y)=\frac{(y,1)}{\sqrt{1+|y|^2}}.
\]
The map $G$ sends shorter spherical geodesic segments in the open
hemisphere to Euclidean line segments, and $P$ sends Euclidean line
segments back to those spherical segments. Hence
$D=G(\Omega)$ is bounded, smooth, and Euclidean convex.

If $N$ is the outward Euclidean unit normal to $\partial D$, put
\[
W=\sqrt{1+|y|^2},\qquad
L=\sqrt{1+(y\cdot N)^2},\qquad
\nu=\frac{(N,-y\cdot N)}{L}.
\]
For tangent vectors $\tau,\sigma\in T_y\partial D$, direct
differentiation yields
\begin{equation}
\label{eq:centralII}
\II_{\Sn}(dP\,\tau,dP\,\sigma)
=\frac1{WL}\II_{\mathbb R^n}(\tau,\sigma).
\end{equation}
Indeed,
\[
dP\,\sigma=\frac{(\sigma,0)}W
-\frac{(y,1)(y\cdot\sigma)}{W^3},
\]
and the derivative of the factor $L^{-1}$ pairs trivially with
$dP\,\sigma$, while the terms containing $(y,1)$ cancel.
Thus $D$ has strictly positive Euclidean principal curvatures.

Choose $\rho>0$ and define the Minkowski interpolation
\[
K_t=(1-t)\overline{B_\rho(0)}+t\overline D,
\qquad D_t=\operatorname{int}K_t,\qquad 0\le t\le1.
\]
Let $h$ be the support function of $\overline D$. The positivity
of the shape operator and strict convexity imply that the Gauss map
is a global diffeomorphism. The inverse Gauss parametrization
$Y(\theta)=h(\theta)\theta+\nabla h(\theta)$ then satisfies
\[
Q_h:=\nabla^2_{\mathbb S^{n-1}}h+h g_{\mathbb S^{n-1}}>0,
\qquad dY=Q_h.
\]
The support function of $K_t$ is
$h_t=(1-t)\rho+th$, so
\begin{equation}
\label{eq:torsionSupportUniform}
Q_{h_t}=(1-t)\rho g_{\mathbb S^{n-1}}+tQ_h,
\qquad
q_*g_{\mathbb S^{n-1}}\le Q_{h_t}\le q^*g_{\mathbb S^{n-1}},
\end{equation}
where one may take
\[
q_*:=\min\{\rho,\min_\theta\lambda_{\min}(Q_h(\theta))\}>0,
\quad
q^*:=\max\{\rho,\max_\theta\lambda_{\max}(Q_h(\theta))\}<\infty.
\]
Consequently the maps
$Y_t(\theta)=h_t(\theta)\theta+\nabla h_t(\theta)$
are smooth embeddings depending smoothly on $(t,\theta)$, and
all Euclidean principal curvatures of $\partial D_t$ lie in
$[(q^*)^{-1},q_*^{-1}]$.

We make the uniform inradius explicit. Choose $y_*\in D$ and
$r_D>0$ with $\overline{B_{r_D}(y_*)}\subset D$. Minkowski addition
gives, for all $t\in[0,1]$,
\begin{equation}
\label{eq:torsionInradiusUniform}
\overline{B_{(1-t)\rho+tr_D}(ty_*)}
=(1-t)\overline{B_\rho(0)}
+t\overline{B_{r_D}(y_*)}
\subset K_t.
\end{equation}
Thus $D_t$ contains a ball of radius at least
$r_*:=\min\{\rho,r_D\}>0$, irrespective of $t$; its center is
allowed to move. Also, because $D$ is bounded, there exists
$R_*<\infty$ with $K_t\subset\overline{B_{R_*}(0)}$ for every $t$.
In particular, the family is uniformly bounded, noncollapsing,
and smooth in the support-function parametrization.

Finally put $\Omega_t=P(D_t)$. On $|y|\le R_*$,
\[
WL\le1+R_*^2,
\qquad
|dP\,\tau|_g^2\le|\tau|^2.
\]
Combining these inequalities with
\eqref{eq:centralII} and the Euclidean curvature bound gives
\[
\II_{\partial\Omega_t}(dP\,\tau,dP\,\tau)
\ge\frac{|\tau|^2}{(1+R_*^2)q^*}
\ge \kappa_*|dP\,\tau|_g^2,
\qquad
\kappa_*:=\frac1{(1+R_*^2)q^*}>0.
\]
Thus the spherical second fundamental form is bounded below
uniformly for the entire path. The boundedness of $K_t$ also keeps
every $\overline\Omega_t$ in one compact subset of the chosen open
hemisphere. Since $P(\overline{B_\rho(0)})$ is a geodesic ball of
radius $\arctan\rho<\pi/2$, the required path is constructed.

For later reference, $q_*,q^*,r_*,R_*$ and $\kappa_*$ depend only
on the fixed initial domain, the hemisphere chart, and the chosen
radius $\rho$; none depends on $t$. The smooth family of boundary
parametrizations, the two-sided bounds for $Q_{h_t}$, and the
compact parameter interval provide uniform bounds for finitely
many boundary charts and their inverses. These are the geometric
data used below to choose a common normal tubular radius and to
apply elliptic estimates uniformly.
\end{proof}

Let $u_t$ denote the torsion function on $\Omega_t$.

\subsection{Uniform estimates along the deformation}

For clarity, we use the following convention throughout this subsection.
The geometric constants $r_*,R_*,q_*,q^*$, $\kappa_*$ and the
uniform tubular radius depend only on the fixed smooth deformation
$\{\Omega_t\}_{0\le t\le1}$ constructed in
Lemma~\ref{lem:domainPath} (including the chosen hemisphere chart
and radius $\rho$). The elliptic constants depend in addition only
on the ambient metric and on a fixed H\"older exponent
$0<\alpha<1$. In particular, the Hopf constant $m$ and the
uniform $C^{2,\alpha}$ bounds are chosen before any boundary
collar width; no constant in this subsection is allowed to depend
on $t$ or on an individual torsion solution $u_t$.

\begin{lemma}
\label{lem:torsionDomainDependence}
Locally in $t$, there are smooth diffeomorphisms
\[
\Psi_t:\Omega_{t_0}\to\Omega_t,
\qquad \Psi_{t_0}=\operatorname{id},
\]
such that $u_t\circ\Psi_t$ is continuous in $C^{2,\alpha}(\overline\Omega_{t_0})$ for every fixed $0<\alpha<1$.  Higher-order local regularity and parameter continuity hold as well.  In particular there are uniform bounds
\begin{equation}
\label{eq:torsionUniformDirichlet}
-\partial_{\nu_t}u_t\ge m>0\quad\text{on }\partial\Omega_t,
\qquad
\sup_t\|u_t\|_{C^{2,\alpha}}<\infty,
\end{equation}
where the norms are measured in a fixed finite family of smoothly varying charts.
\end{lemma}

\begin{proof}
Lemma~\ref{lem:domainPath} constructs smooth boundary embeddings
with a uniform positive lower bound for their second fundamental forms.
Lemma~\ref{lem:quantDomainSmoothness} in
Appendix~\ref{app:technical} constructs the local normal-graph
identifications $\Psi_t$ and proves convergence of the pulled-back
metrics and of the corresponding Dirichlet inverses on the fixed domain
$\Omega_{t_0}$. Applied to $w_t=u_t\circ\Psi_t$, it gives the asserted
$C^{2,\alpha}$ parameter continuity and uniform higher-order bounds.
For each fixed $(t,p)$ with $p\in\partial\Omega_t$, Hopf's lemma
gives $-\partial_{\nu_t}u_t(p)>0$. The pulled-back solutions and
normals depend continuously on $(t,p)$, and the parameterized
boundary is compact; hence the positive Hopf lower bound $m$
may be chosen uniformly in $t$ and $p$. A finite cover of
$[0,1]$ likewise gives uniform $C^{2,\alpha}$ bounds, including
bounds for the gradients and Hessians in the common family of charts.
These are the elliptic constants specified at the beginning
of this subsection, and they are uniform for $t\in[0,1]$.
\end{proof}

\begin{lemma}
\label{lem:torsionCollar}
There exists $\delta>0$, independent of $t$, such that
\[
\Hess(-\sqrt{u_t})>0
\qquad\text{on }\{x\in\Omega_t:\dist(x,\partial\Omega_t)<\delta\}
\]
for every $t\in[0,1]$.
\end{lemma}

\begin{proof}
Write $m>0$ for the uniform Hopf constant in
\eqref{eq:torsionUniformDirichlet} and $\kappa_*>0$ for the
curvature constant in Lemma~\ref{lem:domainPath}. At a boundary
point $p\in\partial\Omega_t$, the tangential boundary condition
$u_t=0$ gives
\begin{equation}
\label{eq:torsionBoundaryHess}
\Hess u_t(p)(w,w)=\partial_{\nu_t}u_t(p)\,
\II_t(w,w)\le-m\kappa_*|w|^2,
\qquad w\in T_p\partial\Omega_t.
\end{equation}
By Lemma~\ref{lem:torsionDomainDependence}, smoothness of the
parameterized boundaries, and compactness of $[0,1]$, there is a
common normal tubular neighborhood of radius $\delta_{\rm tub}>0$.
In this neighborhood choose constants $C\ge1$, $C_1\ge1$ and
$C_H>0$, independent of $t$, such that
\[
|\Hess u_t|\le C,\qquad |\grad u_t|\le C_1,
\]
and the Hessians obey the uniform parallel-transport estimate
\begin{equation}
\label{eq:torsionHessHolder}
\bigl|\Hess u_t(x)-\mathcal P_{p\to x}
(\Hess u_t(p))\bigr|\le C_Hd^\alpha,
\quad d=\dist(x,\partial\Omega_t),
\end{equation}
where $p$ is the nearest boundary point to $x$ and
$\mathcal P_{p\to x}$ transports tensors along the minimizing
normal geodesic. The last estimate follows from the uniform
$C^{2,\alpha}$ bounds in smoothly varying boundary charts;
parallel transport and the charts are uniformly controlled on
the common tube.

We give the quantitative passage from boundary tangent spaces to
nearby level-set tangent spaces. Let $T:T_x\Sn\to T_p\Sn$
be parallel transport, and let $\tau\in T_x\Sn$ satisfy
$\tau\perp\grad u_t(x)$. Along the normal geodesic,
$|\nabla_{\dot\gamma}\grad u_t|\le C$, so
\[
|T\grad u_t(x)-\grad u_t(p)|\le Cd.
\]
Since $\grad u_t(p)=-h_t(p)\nu_t(p)$ with $h_t(p)\ge m$,
orthogonality gives
\[
|\langle T\tau,\nu_t(p)\rangle|
\le\frac{C}{m}d|\tau|.
\]
Decompose $T\tau=w_T+w_\nu\nu_t(p)$, with
$w_T\in T_p\partial\Omega_t$, and use
\eqref{eq:torsionBoundaryHess} together with
$|\Hess u_t(p)|\le C$. We obtain
\begin{align*}
\Hess u_t(x)(\tau,\tau)
&\le -m\kappa_*|\tau|^2
 +\Bigl[C_Hd^\alpha+\frac{2C^2}{m}d
 +(C+m\kappa_*)\frac{C^2}{m^2}d^2\Bigr]|\tau|^2.
\end{align*}
Choose $0<\delta_0<\delta_{\rm tub}$ small enough that
\begin{equation}
\label{eq:torsionCollarDelta0}
C_H\delta_0^\alpha+\frac{2C^2}{m}\delta_0
 +(C+m\kappa_*)\frac{C^2}{m^2}\delta_0^2
\le\frac{m\kappa_*}{2},
\qquad
C\delta_0\le\frac m2.
\end{equation}
Therefore, throughout $0<d<\delta_0$,
\begin{equation}
\label{eq:torsionLevelTangent}
|\grad u_t|\ge a:=\frac m2,
\qquad
\Hess u_t(\tau,\tau)\le-b|\tau|^2
\quad(\tau\perp\grad u_t),
\qquad b:=\frac{m\kappa_*}{2}.
\end{equation}
Both the tube and the constants are independent of $t$.

Suppress $t$ and let $N=\grad u/|\grad u|$. Write
$X=\tau+sN$, with $\tau\perp N$. From
\eqref{eq:torsionChain}, \eqref{eq:torsionLevelTangent}, and
$|\Hess u|\le C$,
\[
\begin{aligned}
4u^{3/2}\Hess(-\sqrt u)(X,X)
&=du(X)^2-2u\Hess u(X,X)\\
&\ge 2bu|\tau|^2-4Cu|s||\tau|+(a^2-2Cu)s^2\\
&\ge bu|\tau|^2+
\left[a^2-\left(2C+\frac{4C^2}{b}\right)u\right]s^2.
\end{aligned}
\]
Along a minimizing normal segment from $p\in\partial\Omega_t$
to $x$, the fundamental theorem of calculus and
$|\grad u_t|\le C_1$ give
\[
0<u_t(x)\le C_1\dist(x,\partial\Omega_t).
\]
Thus we may fix \emph{one} width $\delta$, valid for the entire
family, satisfying
\begin{equation}
\label{eq:torsionExplicitCollarWidth}
0<\delta<\min\left\{\delta_0,\delta_{\rm tub},
\frac{a^2}{2C_1(2C+4C^2/b)}\right\}.
\end{equation}
The square bracket is then at least $a^2/2$, and hence
\begin{equation}
\label{eq:torsionCollarEstimate}
\Hess(-\sqrt u)(X,X)
\ge\frac{b}{4\sqrt u}|\tau|^2+
\frac{a^2}{8u^{3/2}}s^2>0
\qquad(X\ne0).
\end{equation}
The constants $m,\kappa_*,C,C_1,C_H$ and $\delta_{\rm tub}$
are fixed uniformly along the domain path. The order of choices
is $\delta_0$ from \eqref{eq:torsionCollarDelta0}, then
$a=m/2$ and $b=m\kappa_*/2$, and finally $\delta$ from
\eqref{eq:torsionExplicitCollarWidth}. These choices depend only
on the domain path and the elliptic estimates, not on any
prior convexity of the transformed solutions.
\end{proof}

\subsection{Continuation along the domain path}

\begin{lemma}
\label{lem:torsionBall}
If $\Omega_0=B_R(p)$ with $R<\pi/2$, then its torsion function satisfies
\[
\Hess u_0<0,
\qquad
\Hess\sqrt{u_0}<0
\qquad\text{in }\Omega_0.
\]
\end{lemma}

\begin{proof}
Uniqueness and rotational symmetry give $u_0=U(r)$, where $r=\dist(p,\cdot)$ and
\[
U''+(n-1)\cot r\,U'=-1,
\qquad U(R)=0,
\qquad U'(0)=0.
\]
Put
\[
J(r)=\int_0^r\sin^{n-1}s\,ds.
\]
Then
\[
U'(r)=-\frac{J(r)}{\sin^{n-1}r}<0.
\]
Since $\cos s\ge\cos r>0$ for $0\le s\le r<R$,
\[
\cos r\,J(r)
\le\int_0^r\sin^{n-1}s\cos s\,ds
=\frac{\sin^n r}{n}.
\]
Consequently
\[
U''=-1+(n-1)\frac{\cos r\,J(r)}{\sin^n r}
\le-\frac1n.
\]
For $r>0$,
\[
\Hess u_0
=U''dr\otimes dr+U'\cot r\,(g-dr\otimes dr)<0.
\]
At the center, $J(r)/\sin^{n-1}r=r/n+O(r^3)$ gives $\Hess u_0(p)=-g/n$.  Finally,
\[
\Hess\sqrt{u_0}
=\frac{\Hess u_0}{2\sqrt{u_0}}
-\frac{du_0\otimes du_0}{4u_0^{3/2}}<0.
\]
\end{proof}

\begin{proof}[Proof of Theorem~\ref{thm:torsion}]
Choose the family from Lemma~\ref{lem:domainPath} and define
\[
I=\{t\in[0,1]:\Hess(-\sqrt{u_t})>0\text{ throughout }\Omega_t\}.
\]
It is nonempty by Lemma~\ref{lem:torsionBall}.

It is relatively closed.  If $t_j\in I$ and $t_j\to t$, use Lemma~\ref{lem:torsionDomainDependence} on an arbitrary compact subset of $\Omega_t$.  There $u_t$ is bounded away from zero, so composition with $-\sqrt{\cdot}$ preserves the required $C^2$ convergence.  Hence
\[
\Hess(-\sqrt{u_t})\ge0
\qquad\text{throughout }\Omega_t.
\]
Lemma~\ref{lem:torsionStrict} upgrades this to strict positivity, so $t\in I$.

It is relatively open.  Otherwise there exist $t_j\to t_0\in I$, points $x_j\in\Omega_{t_j}$, and unit vectors $X_j\in T_{x_j}\Sn$ such that
\[
\Hess(-\sqrt{u_{t_j}})(X_j,X_j)\le0.
\]
Lemma~\ref{lem:torsionCollar} gives
\[
\dist(x_j,\partial\Omega_{t_j})\ge\delta.
\]
Pulling back by the local identifications at $t_0$, a subsequence converges to an interior point and a nonzero tangent vector.  Interior $C^2$ convergence then contradicts the strict positivity at $t_0$.  Thus $I$ is open.  Connectedness of $[0,1]$ gives $I=[0,1]$, and changing the sign proves \eqref{eq:torsionconc}.
\end{proof}

\begin{proof}[Proof of Corollary~\ref{cor:torsionConsequences}]
Put $w=\sqrt u$.  Since $u>0$ in $\Omega$ and vanishes on the boundary, it has an interior maximum.  If $x_0$ and $x_1$ were two distinct critical points, the minimizing geodesic between them lies in $\Omega$ by geodesic convexity.  Along its unit-speed parametrization $\gamma$,
\[
\frac{d^2}{ds^2}(w\circ\gamma)
=\Hess w(\dot\gamma,\dot\gamma)<0.
\]
The first derivative is strictly decreasing and therefore cannot vanish at both endpoints.  Thus the critical point is unique.  At that point $dw=0$ and
\[
\Hess u=2w\Hess w<0,
\]
so the maximum is nondegenerate.

For $0<c<\max_\Omega u$, the level set $\{u=c\}$ is regular and compactly contained in $\Omega$.  Concavity of $w$ along minimizing geodesics shows that $\{u>c\}$ is geodesically convex.  Its outward unit normal is $-\grad w/|\grad w|$, and for tangent $X\ne0$,
\[
\II_{\partial\{u>c\}}(X,X)
=-\frac{\Hess w(X,X)}{|\grad w|}>0.
\]

Finally, the collar estimate \eqref{eq:torsionCollarEstimate} gives a positive lower bound for $\Hess(-\sqrt u)$ in a fixed boundary collar.  On the complementary compact subset, Theorem~\ref{thm:torsion} and continuity give another positive lower bound.  Taking their minimum and reversing the sign yields \eqref{eq:torsionUniform}.
\end{proof}

The normalization to unit sectional curvature is inessential.  Let $\mathbb S^n_\kappa$ be the round sphere of sectional curvature $\kappa>0$, and let $u$ solve $-\Delta_{g_\kappa}u=\lambda$ with $\lambda>0$.  Under the homothety $D_\kappa(x)=\kappa^{-1/2}x$, the pulled-back function
\[
U=\frac\kappa\lambda(u\circ D_\kappa)
\]
solves $-\Delta_gU=1$, and
\[
D_\kappa^*\!\left(\nabla^2_{g_\kappa}\sqrt u\right)
=\sqrt{\frac\lambda\kappa}\,\Hess\sqrt U<0.
\]
Consequently, Theorem~\ref{thm:torsion} extends by rescaling to round spheres of arbitrary positive sectional curvature and to any positive constant forcing term. The preceding proof uses the curvature identities of the sphere and does not directly establish an analogous statement on a general Riemannian manifold.

\section{An extension to non-power nonlinearities}
\label{sec:nonpower}

\subsection{An intrinsic strictification criterion}
\label{sec:intrinsic}

Consider
\begin{equation}
\label{eq:generalf}
-\Lap u=f(u)\qquad\text{in }U\subset\Sn,
\end{equation}
where $u>0$.  We assume
\begin{equation}
\label{eq:H0}
 f\in C^3((0,\infty)),\quad f>0,
\qquad
\int_0^1\frac{ds}{f(s)}=\infty,
\qquad
\int_1^\infty\frac{ds}{f(s)}<\infty.
\end{equation}
Define
\begin{equation}
\label{eq:Hdef}
H(s):=\int_s^\infty\frac{d\tau}{f(\tau)},
\qquad
v:=H(u)^{1/3}.
\end{equation}
Consequently, $H:(0,\infty)\to(0,\infty)$ is strictly decreasing, $H(s)\to\infty$ as $s\downarrow0$, and $H(s)\to0$ as $s\to\infty$.  Since $H'(u)=-1/f(u)$ and $H(u)=v^3$,
\begin{equation}
\label{eq:vuprime}
\frac{dv}{du}=-\frac{1}{3f(u)v^2}<0.
\end{equation}
For the local theorem below, the divergence of $H$ at zero can be weakened to finiteness and positivity of $H$ on the range of the particular solution; the full assumption \eqref{eq:H0} is natural for the global Dirichlet theory.

Define, on the range of $v$,
\begin{equation}
\label{eq:Gammadef}
\Gamma(v):=3f'(u)H(u)-2.
\end{equation}
The structural conditions entering the local strictification theorem are
\begin{equation*}
\tag{H1}
\Gamma(v)\ge2(n-1)
\end{equation*}
and
\begin{equation*}
\tag{H2}
\frac{d}{dv}\bigl(v^2\Gamma'(v)\bigr)\le0,
\end{equation*}
equivalently $\Gamma''(v)+2\Gamma'(v)/v\le0$.

\begin{theorem}\label{thm:IntrinsicSchur}
Let $n\ge2$, let $U\subset\Sn$ be open, and let
$f\in C^3((0,\infty))$ be positive. Suppose that
$u\in C^4(U)$ solves $-\Lap u=f(u)$ and that $H$ is finite and
positive on the range of $u$. Let $v=H(u)^{1/3}$ and assume, throughout
the range of $v$, that
\begin{equation}\label{eq:introGamma}
\Gamma(v)\ge2(n-1),\qquad
\frac{d}{dv}\bigl(v^2\Gamma'(v)\bigr)\le0.
\end{equation}
If $\Hess v\ge0$ in $U$, then
\[
\Hess v>0\quad\text{in }U.
\]
\end{theorem}

\begin{proof}[Proof of Theorem~\ref{thm:IntrinsicSchur}]
From \eqref{eq:vuprime},
\[
-\frac{dv}{du}f(u)=\frac{1}{3v^2}.
\]
Moreover,
\[
\frac{v''(u)}{(v'(u))^2}
=3f'(u)v^2-\frac2v
=\frac{3f'(u)H(u)-2}{v}
=\frac{\Gamma(v)}v.
\]
Thus
\begin{equation}
\label{eq:IntrinsicPDE}
F:=\Lap v
=\frac{\Gamma(v)}v q+\frac{1}{3v^2},
\qquad q:=|\grad v|^2.
\end{equation}
In particular $F>0$ under \textup{(H1)}.

Suppose that $\Hess v$ has a null direction at $x_0$, of rank $r$.  Apply Lemma~\ref{lem:SchurMaster} with
\[
\Phi(v,q)=\frac{\Gamma(v)}v q+\frac{1}{3v^2}.
\]
Writing
\[
R:=r+\frac{\Gamma'(v)}v q-\frac{1}{3v^3},
\]
one checks from \eqref{eq:IntrinsicPDE} that
\[
r+\Phi_v=R-\frac{F}{v}.
\]
Expanding the square in \eqref{eq:MasterSchur} and using \eqref{eq:IntrinsicPDE} yields
\begin{equation}
\label{eq:IntrinsicSchurFormula}
\begin{aligned}
\Delta\phi(x_0)
\le{}&
\Bigg[
\frac{4r-2\Gamma(v)}v
+q\left(
\frac{\Gamma''(v)}v
+\frac{2\Gamma'(v)}{v^2}
\right)\\
&\qquad
-\frac2F
\left(
 r+\frac{\Gamma'(v)}v q-\frac{1}{3v^3}
\right)^2
\Bigg]v_1^2
-\frac{2}{3v^2}.
\end{aligned}
\end{equation}
Since $r\le n-1$, condition \textup{(H1)} makes the first term in the bracket nonpositive.  Condition \textup{(H2)} makes the second nonpositive, the square term is nonpositive, and the final term is strictly negative.  Therefore $\Delta\phi(x_0)<0$, contradicting \eqref{eq:masterMin}.
\end{proof}

\begin{corollary}
\label{cor:IntrinsicFConditions}
Condition \textup{(H1)} is equivalent to
\begin{equation}
\label{eq:fIntrinsic1}
f'(u)H(u)\ge\frac{2n}{3}.
\end{equation}
Furthermore,
\begin{equation}
\label{eq:GammaPrimeF}
\Gamma'(v)=9v^2\left[f'(u)-f(u)f''(u)H(u)\right],
\end{equation}
and \textup{(H2)} is equivalent to
\begin{equation}
\label{eq:fIntrinsic2}
\frac{d}{du}
\left[
H(u)^{4/3}
\bigl(f'(u)-f(u)f''(u)H(u)\bigr)
\right]\ge0.
\end{equation}
\end{corollary}

\begin{proof}
The first assertion follows from \eqref{eq:Gammadef}.  Differentiating $\Gamma=3f'H-2$ with respect to $u$ and using
\[
H'=-\frac1f,
\qquad
\frac{du}{dv}=-3fv^2
\]
gives \eqref{eq:GammaPrimeF}.  Finally,
\[
\frac{d}{dv}(v^2\Gamma')\le0
\quad\Longleftrightarrow\quad
\Gamma''+\frac2v\Gamma'\le0,
\]
and
\[
v^2\Gamma'=9v^4(f'-ff''H)=9H^{4/3}(f'-ff''H).
\]
Since $du/dv<0$, the sign condition in $v$ is precisely \eqref{eq:fIntrinsic2}.
\end{proof}

\begin{remark}
Conditions \textup{(H1)}--\textup{(H2)} are the structural assumptions used in the proof. The first controls the curvature term for all possible ranks $r\le n-1$ of the transformed Hessian. The second can also be written as the differential inequality
\[
\Gamma''(v)+\frac2v\Gamma'(v)\le0.
\]
The criterion is unchanged under multiplication of the reaction by a positive constant, after the corresponding rescaling of the transformed variable.
\end{remark}

\begin{remark}
For $f(u)=u^p$,
\[
H(u)=\frac{u^{1-p}}{p-1},
\qquad
f'(u)H(u)=\frac{p}{p-1},
\]
and
\[
f'(u)-f(u)f''(u)H(u)\equiv0.
\]
Thus \eqref{eq:fIntrinsic2} holds with equality, and \eqref{eq:fIntrinsic1} becomes
\[
\frac{p}{p-1}\ge\frac{2n}{3}
\quad\Longleftrightarrow\quad
p\le\frac{2n}{2n-3}.
\]
Moreover $H(u)^{1/3}$ differs from $u^{-(p-1)/3}$ only by the positive constant $(p-1)^{-1/3}$.
\end{remark}

\subsection{Application to an inhibited power law}
\label{sec:inhibited}

We first treat a nonhomogeneous reaction which is specified directly,
rather than obtained by inverting a prescribed intrinsic coefficient.
For $p>1$, $\lambda\ge0$, and $0<\delta<p-1$, define
\begin{equation}\label{eq:inhibitedDef}
f_{p,\lambda,\delta}(u):=\frac{u^p}{1+\lambda u^\delta},
\qquad u>0.
\end{equation}
Its intrinsic tail has the explicit form
\begin{equation}\label{eq:inhibitedH}
H(u)=\int_u^\infty \frac{ds}{f_{p,\lambda,\delta}(s)}
=\frac{u^{1-p}}{p-1}+
\frac{\lambda u^{1-p+\delta}}{p-1-\delta}.
\end{equation}

\begin{proposition}
\label{prop:inhibitedStructure}
For $n\ge2$ and
\begin{equation}\label{eq:inhibitedStructureRange}
1<p\le\frac{2n}{2n-3},\qquad
\lambda\ge0,\qquad 0<\delta\le\frac{p-1}{3},
\end{equation}
let $v=H(u)^{1/3}$ and $\Gamma(v)=3f'(u)H(u)-2$.
Then, for every $u>0$,
\[
\Gamma(v)\ge2(n-1),\qquad
\frac{d}{dv}\bigl(v^2\Gamma'(v)\bigr)\le0.
\]
\end{proposition}
\begin{proof}
Put $A=p-1$, $B=A-\delta$, and $z=\lambda u^\delta$.
In particular, $A>0$ and $B\ge2A/3>0$. A direct calculation gives
\begin{equation}\label{eq:inhibitedQ}
Q(z):=f'(u)H(u)
=\frac{[p+(p-\delta)z](A^{-1}+zB^{-1})}{(1+z)^2},
\qquad
Q'(z)=\frac{\delta[(1+\delta)+(1-\delta)z]}{AB(1+z)^3}.
\end{equation}
Since $A\le3$ and $\delta\le A/3$, we have $\delta\le1$;
hence $Q'(z)\ge0$. Consequently
\[
f'(u)H(u)=Q(z)\ge Q(0)=\frac p{p-1}\ge\frac{2n}{3},
\]
which is equivalent to the first structural inequality by
Corollary~\ref{cor:IntrinsicFConditions}.

We next verify the second inequality using \eqref{eq:fIntrinsic2}.
As $Q=f'H$ and $H'=-1/f$,
\[
f'-ff''H=-f\frac{dQ}{du}
=-\frac{\delta^2u^A z[(1+\delta)+(1-\delta)z]}
{AB(1+z)^4}.
\]
For $\lambda=0$ this expression vanishes and the required
condition follows. If $\lambda>0$, multiplication by $H^{4/3}$
shows that for some positive constant $C$ depending only on the
parameters,
\begin{equation}\label{eq:inhibitedWidentity}
H^{4/3}(f'-ff''H)=-C W(z),
\quad
W(z)=z^{1-A/(3\delta)}
\frac{[(1+\delta)+(1-\delta)z](B+Az)^{4/3}}{(1+z)^4}.
\end{equation}
Logarithmic differentiation gives
\begin{align*}
\frac{zW'(z)}{W(z)}={}&1-\frac{A}{3\delta}
+\frac{(1-\delta)z}{1+\delta+(1-\delta)z}
+\frac{4Az}{3(B+Az)}-\frac{4z}{1+z}.
\end{align*}
The first term is nonpositive. Using $0<\delta\le1$ and
$B\ge2A/3$, we have
\[
\frac{(1-\delta)z}{1+\delta+(1-\delta)z}
\le\frac{z}{1+z},\qquad
\frac{4Az}{3(B+Az)}\le\frac{2z}{1+z}.
\]
Therefore
\[
\frac{zW'(z)}{W(z)}\le-\frac{z}{1+z}<0
\qquad(z>0).
\]
Since $z=\lambda u^\delta$ increases with $u$, the right-hand
side of \eqref{eq:inhibitedWidentity} increases with $u$.
Thus \eqref{eq:fIntrinsic2} holds, completing the proof.
\end{proof}

The preceding proposition is a \emph{local strictification result}:
it does not, by itself, produce a solution with a semidefinite
transformed Hessian. The next elementary computation supplies the
extra information needed to extend the global compactness proof
beyond the half-space supersolution exponent.

\begin{lemma}
\label{lem:inhibitedConvexReaction}
Assume $1<p\le p_n^*=2n/(2n-3)$ and
$0<\delta\le(p-1)/3$. For any $\lambda\ge0$, the function
$f(s)=s^p/(1+\lambda s^\delta)$ is strictly convex on $(0,\infty)$,
nonnegative, and belongs to
$C^1([0,\infty))\cap C^2((0,\infty))$ with $f(0)=0$.
The same assertions hold for every normalized multiple
$h(s)=C s^p/(1+\lambda s^\delta)$ with $C>0$,
and for the limiting power $h(s)=s^{p-\delta}$.
\end{lemma}
\begin{proof}
Put $A=p-1\in(0,3]$, $z=\lambda s^\delta$, and
$w=z/(1+z)\in[0,1)$. Let
\[
E(s):=\frac{s f'(s)}{f(s)}=p-\delta w.
\]
Since $sE'(s)=-\delta^2w(1-w)$, differentiation gives
\begin{equation}\label{eq:inhibitedConvexCalculation}
\frac{s^2f''(s)}{f(s)}
=E(s)(E(s)-1)-\delta^2w(1-w).
\end{equation}
Because $\delta\le A/3$, one has
$E\ge1+2A/3$. Also $w(1-w)\le1/4$, so
\[
\frac{s^2f''(s)}{f(s)}
\ge\frac{2A}{3}-\frac{A^2}{36}
\ge\frac{7A}{12}>0,
\]
where the last inequality uses $A\le3$.
The differentiability assertions follow from $p-\delta>1$.
Multiplication by a positive constant preserves convexity,
and $p-\delta>1$ makes the limiting power strictly convex.
\end{proof}

The global continuation argument uses the convex-reaction
half-space Liouville theorem of Dupaigne--Sirakov--Souplet
\cite[Theorem~1.3]{DupaigneSirakovSouplet2022}; the
hypotheses on the limiting reaction are verified below.

\begin{proof}[Proof of Theorem~\ref{thm:inhibited}]
Let
\[
P:=p_n^*=\frac{2n}{2n-3},
\qquad \theta:=\frac{\delta}{p-1}\in(0,1/3].
\]
Choose $p_*\in(1,p)$ sufficiently close to $1$ that the
near-linear uniqueness, nondegeneracy, and strict power-convexity
result of Subsection~\ref{sec:near1} holds at $p_*$. For
$t\in[0,1]$ set
\begin{equation}\label{eq:inhibitedHomotopy}
p_t=(1-t)p_*+tp,\quad
\delta_t=\theta(p_t-1),\quad
f_t(u)=\frac{u^{p_t}}{1+t\lambda u^{\delta_t}}.
\end{equation}
Then $f_0(u)=u^{p_*}$ and $f_1(u)=f_{p,\lambda,\delta}(u)$.
Proposition~\ref{prop:inhibitedStructure} applies at each $t$,
including $t=0$.

\emph{Uniform compactness.}
The effective exponent of $f_t$ is
\begin{equation}\label{eq:inhibitedElasticity}
P_t(u):=\frac{uf_t'(u)}{f_t(u)}
=p_t-\frac{\delta_t\,t\lambda u^{\delta_t}}
{1+t\lambda u^{\delta_t}}.
\end{equation}
With $\alpha:=1+(1-\theta)(p_*-1)>1$, it follows that
\begin{equation}\label{eq:inhibitedElasticityBounds}
1<\alpha\le P_t(u)\le p\le P=p_n^*
\quad(u>0,\ t\in[0,1]).
\end{equation}
For $M\ge1$, we also have
\[
\frac{f_t(M)}M\ge
\frac{M^{(1-\theta)(p_*-1)}}{1+\lambda}\longrightarrow\infty
\quad\text{uniformly in }t,
\]
whereas for $0<u\le1$,
$f_t(u)/u\le u^{p_*-1}\to0$ uniformly as $u\downarrow0$.
Suppose positive Dirichlet solutions $u_j$ along this homotopy
have maxima $M_j\to\infty$. Put
\[
\rho_j^2=\frac{M_j}{f_{t_j}(M_j)},\qquad
h_j(s)=\frac{f_{t_j}(M_js)}{f_{t_j}(M_j)}
\quad(0\le s\le1).
\]
Then $\rho_j\to0$, and integration of \eqref{eq:inhibitedElasticityBounds}
gives
\begin{equation}\label{eq:inhibitedNormalized}
s^p\le h_j(s)\le s,\qquad
0\le h_j'(s)\le p\quad(0<s\le1).
\end{equation}
We identify the nonlinear limit before applying the
convex-reaction half-space Liouville theorem. Let
\[
  z_j:=t_j\lambda M_j^{\delta_{t_j}}\ge0.
\]
The normalized reaction has the \emph{exact} expression
\begin{equation}\label{eq:inhibitedNormalizedExact}
 h_j(s)=s^{p_{t_j}}\frac{1+z_j}{1+z_js^{\delta_{t_j}}}
 \qquad(0\le s\le1).
\end{equation}
After taking a subsequence, assume that $t_j\to t_\infty$
and $z_j\to z_\infty\in[0,\infty]$ (allowing $z_\infty=\infty$).
Write $p_\infty=p_{t_\infty}$ and
$\delta_\infty=\theta(p_\infty-1)$. For every $s>0$,
\begin{equation}\label{eq:inhibitedLimitTwoCases}
 h_j(s)\longrightarrow h(s):=
 \begin{cases}
 \displaystyle s^{p_\infty}\frac{1+z_\infty}
                   {1+z_\infty s^{\delta_\infty}},
                   &z_\infty<\infty,\\[2mm]
 s^{p_\infty-\delta_\infty},&z_\infty=\infty.
 \end{cases}
\end{equation}
Set $h(0)=0$. The uniform Lipschitz estimate in
\eqref{eq:inhibitedNormalized} and pointwise convergence imply
$h_j\to h$ uniformly on $[0,1]$. In either case the formula in
\eqref{eq:inhibitedLimitTwoCases} defines $h$ on all of
$[0,\infty)$; Lemma~\ref{lem:inhibitedConvexReaction} shows that
\[
 h\in C^1([0,\infty))\cap C^2((0,\infty)),\quad
 h(0)=0,\quad h\ge0,\quad h\text{ is strictly convex},\quad
 h\not\equiv0.
\]
In particular $h(s)\ge s^p$ for $0\le s\le1$.

Let $x_j$ maximize $u_j$ and put $d_j=\dist(x_j,\partial\Omega)$.
For $w_j=u_j/M_j$, the original-domain equation is
$-\Lap w_j=\rho_j^{-2}h_j(w_j)$, and $0\le w_j\le1$, $w_j(x_j)=1$.
By \eqref{eq:inhibitedNormalized} the $h_j$ are equi-Lipschitz
on $[0,1]$; hence the hypotheses of
Lemma~\ref{lem:quantBoundaryBlowup} hold with a fixed $L$.
After extraction, either $d_j/\rho_j\to\infty$, or
$d_j/\rho_j\to d_\infty<\infty$.
In the first case, rescaling in exponential coordinates centered
at $x_j$ and using interior $W^{2,q}$ and Schauder estimates yields
$w_j\to W$ locally in $C^2(\mathbb R^n)$, with $W(0)=1$.
In the second case the boundary lemma, applied in normal coordinates,
gives $d_\infty\ge c(\Omega,n,D,L)>0$, where
$d_j/\rho_j\le D<\infty$ on the chosen subsequence, and a
$C^2_{\mathrm{loc}}$ limit up to the flat boundary with
$W(d_\infty e_n)=1$ and $W=0$ on $\partial\mathbb R^n_+$.
The uniform convergence $h_j\to h$ shows in either case that
\begin{equation}\label{eq:inhibitedBlowupLimit}
 -\Delta W=h(W),\qquad 0\le W\le1.
\end{equation}
Since $h\ge0$ and the maximum is attained at an interior point,
the strong maximum principle yields $W>0$ throughout the relevant
open domain. In the half-space case $W$ also has its continuous zero
boundary trace and is bounded on every finite strip.

\emph{Whole-space alternative.}
If $n=2$, a nonnegative superharmonic function on $\mathbb R^2$
is constant. Equation \eqref{eq:inhibitedBlowupLimit} and
$h(s)>0$ for $s>0$ then force $W=0$, a contradiction.
If $n\ge3$, the lower bound $h(W)\ge W^p$ gives
$-\Delta W\ge W^p$ in $\mathbb R^n$.
The inequality
\[
 p\le\frac{2n}{2n-3}<\frac{n}{n-2}
\]
exclude such a positive entire supersolution: by
Armstrong--Sirakov \cite[Theorem~4.1]{ArmstrongSirakov2011},
with $-\Delta$ as the operator and the pure-power function
$f(s)=s^p$, no positive supersolution of
$-\Delta W\ge W^p$ exists on an exterior domain when
$1<p\le n/(n-2)$. We apply this result after restricting
$W$ to the exterior of a ball.

\emph{Half-space alternative.}
The limit is a nontrivial nonnegative classical Dirichlet
solution of \eqref{eq:inhibitedBlowupLimit}, bounded by one on
the entire half-space and thus on every finite-height strip.
The general convex-reaction Liouville theorem of
Dupaigne--Sirakov--Souplet \cite[Theorem~1.3]{DupaigneSirakovSouplet2022},
applies because $h\in C^1([0,\infty))\cap C^2((0,\infty))$,
$h(0)=0$, $h\ge0$, $h\not\equiv0$, and $h$ is convex.
The limit $W$ is nonnegative, nonzero, continuous up to the flat
boundary with zero Dirichlet trace, and bounded on finite-height
strips. The cited theorem therefore rules out this alternative in \emph{every}
dimension. In particular, no exponent restriction
$p\le(n+1)/(n-1)$ is required.

Both alternatives being impossible, the homotopy family has a
uniform $L^\infty$ bound for the full range $1<p\le p_n^*$.
For the lower bound, $p_t\ge p_*$ gives
$0<f_t(s)/s\le s^{p_*-1}$ for $0<s\le1$, uniformly in $t$.
Choose $0<\eta<1$ with $\eta^{p_*-1}<\lambda_1$.
If a positive solution had $\|u\|_\infty\le\eta$, testing
against $\phi_1>0$ would give
$\lambda_1\int u\phi_1=\int f_t(u)\phi_1
\le\eta^{p_*-1}\int u\phi_1$, a contradiction.
Hence $\eta\le\|u\|_\infty\le M_*$, with $M_*$ independent
of $t$. On $0\le s\le M_*$,
\eqref{eq:inhibitedElasticityBounds} implies
$0\le f_t'(s)\le p\max_{t}f_t(M_*)/M_*$;
the right-hand side is finite by joint continuity, and the
one-sided derivative at zero is zero. Therefore
Lemma~\ref{lem:parameterCompactness} applies to the entire
homotopy family: for each fixed $\alpha'<1$ the positive solutions
form a compact set in $[0,1]\times C^{2,\alpha'}(\overline\Omega)$,
with uniform Hopf, gradient, and tangential-Hessian bounds in a
common boundary collar.

\emph{Continuation.}
Consider the completely continuous fixed-point map
\[
T(t,u)=(-\Lap)^{-1}f_t(u_+)
\qquad\text{on }X=C_0(\overline\Omega).
\]
The upper and lower bounds just obtained provide a common bounded
annulus containing every nonzero positive fixed point and no fixed
points on its boundary. At $t=0$, the near-linear positive solution
is unique and nondegenerate. Its Leray--Schauder degree is $\pm1$.
The map is jointly continuous and compact on bounded sets,
and no fixed point lies on the boundary of the annulus for any
$t\in[0,1]$. At $t=0$ the fixed-point set inside the annulus
is the singleton comprising the unique near-linear solution;
its nondegeneracy gives degree $\pm1$. Thus the hypotheses of Mawhin's continuation theorem
\cite[Theorem~1]{Mawhin1997} hold: complete continuity,
no fixed points on the lateral boundary, a singleton initial
fixed-point set, and nonzero degree. This yields a compact connected
set $\mathscr C_{\mathrm{inh}}$ of positive solutions whose
parameter projection is $[0,1]$. No convexity of the annulus
is needed.

\emph{A common convex boundary collar.}
Write $H_t(u)=\int_u^\infty f_t(s)^{-1}ds$ and
$\Phi_t(u)=H_t(u)^{1/3}$. Integrating the lower bound in
\eqref{eq:inhibitedElasticityBounds} gives
\begin{equation}\label{eq:inhibitedProduct}
f_t(u)H_t(u)\le\frac{u}{\alpha-1},\qquad
\frac{\Phi_t''(u)}{-\Phi_t'(u)}
=\frac{\Gamma_t(\Phi_t(u))}{3f_t(u)H_t(u)}
\ge\frac{2(n-1)(\alpha-1)}{3u}.
\end{equation}
The first estimate also implies finiteness and uniform decay of
the tails $H_t(u)$ at infinity. Write $L_t(u)=-\Phi_t'(u)>0$ and decompose any unit vector
$X=\tau+sN$, with $N=\grad u/|\grad u|$ and
$\tau\perp\grad u$. By Lemma~\ref{lem:parameterCompactness},
there are constants $a_0,b_0,C_0,c_1>0$ and a collar width
$\delta_0>0$, independent of $t$ and the positive solution, with
$|\grad u|\ge a_0$, $|\Hess u|\le C_0$,
$\Hess u(\tau,\tau)\le-b_0|\tau|^2$ and
$u(x)\le c_1\dist(x,\partial\Omega)$ there. Consequently,
using Young's inequality on the mixed Hessian term,
\[
\begin{aligned}
\Hess(\Phi_t(u))(X,X)
&=L_t[-\Hess u(X,X)+(\Phi_t''/L_t)du(X)^2]\\
&\ge L_t\left\{\frac{b_0}{2}|\tau|^2+
\left[a_0^2\frac{\Phi_t''}{L_t}-C_0-
\frac{2C_0^2}{b_0}\right]s^2\right\}.
\end{aligned}
\]
By \eqref{eq:inhibitedProduct}, the bracket is at least
$2(n-1)(\alpha-1)a_0^2/(3u)-C_0-2C_0^2/b_0$.
Choose $0<\delta\le\delta_0$ so that this expression
is positive for $0<u\le c_1\delta$.
This yields $\Hess(\Phi_t(u))>0$ whenever
$\dist(x,\partial\Omega)<\delta$, uniformly on
$\mathscr C_{\mathrm{inh}}$.

\emph{Propagation of strict convexity.}
The strict-convexity locus in $\mathscr C_{\mathrm{inh}}$ is relatively
open by the common boundary collar and interior $C^2$ convergence.
It is relatively closed by the same convergence, followed by
Theorem~\ref{thm:IntrinsicSchur} and
Proposition~\ref{prop:inhibitedStructure}. At $t=0$ it contains
the near-linear strictly power-convex solution, because
$H_0(u)^{1/3}=(p_*-1)^{-1/3}u^{-(p_*-1)/3}$.
Connectedness therefore makes the entire continuum strictly
intrinsically convex. Taking $t=1$ gives
\eqref{eq:inhibitedConvIntro} by \eqref{eq:inhibitedH},
with the asserted geometric consequences.
\end{proof}

\clearpage
\appendix
\section{Uniform compactness and domain-dependence estimates}
\label{app:technical}

This appendix gives the compactness estimates for parameter-dependent Dirichlet problems, the boundary estimates used in blow-up arguments, and the regularity needed for smooth domain deformations. The dependence of constants is stated in each result.

\subsection{Parameter-uniform boundary estimates}

\begin{lemma}
\label{lem:parameterCompactness}
Let $\Omega\subset\mathbb S^n$ be a fixed smooth uniformly
geodesically convex domain. Let $f_t:[0,\infty)\to[0,\infty)$,
$t\in[0,1]$, satisfy $f_t(0)=0$, $f_t(s)>0$ for $s>0$, and
$(t,s)\mapsto f_t(s)$ be continuous. Assume that for each $M>0$
there is $L_M<\infty$ with
\[
 |f_t(s)-f_t(r)|\le L_M|s-r|
 \quad(0\le r,s\le M,\;0\le t\le1).
\]
Suppose every positive classical Dirichlet solution of
$-\Delta_g u=f_t(u)$ satisfies
\begin{equation}\label{eq:parameterBounds}
 0<m_*\le\|u\|_{\infty}\le M_*<\infty
\end{equation}
with the same constants for all $t\in[0,1]$.
Then the set $\mathscr P$ of all such pairs $(t,u)$ is compact
in $[0,1]\times C^{2,\alpha'}(\overline\Omega)$ for every fixed
$0<\alpha'<1$. Moreover there exist positive constants
$a_0,b_0,C_0,c_1,\delta_0$, depending only on $\Omega$, the
bounds in \eqref{eq:parameterBounds}, $L_{M_*}$, and the modulus
of continuity of $(t,s)\mapsto f_t(s)$ on
$[0,1]\times[0,M_*]$, such that all these solutions satisfy
\begin{align*}
 |\nabla_g u|&\ge a_0, &|\nabla_g^2u|&\le C_0,\\
 \nabla_g^2u(\tau,\tau)&\le-b_0|\tau|^2
 &\text{for }\tau\perp\nabla_g u
\end{align*}
when $\dist(x,\partial\Omega)<\delta_0$, and
\[
 0\le u(x)\le c_1\dist(x,\partial\Omega)
 \quad\text{for }\dist(x,\partial\Omega)<\delta_0.
\]
In particular, $-\partial_\nu u\ge2a_0$ after possibly
reducing $a_0$.
\end{lemma}

\begin{proof}
The bounds $\|u\|_\infty\le M_*$ and continuity of $f_t$
on the compact parameter--amplitude rectangle give a common
$L^\infty$ bound for the right-hand sides. On the fixed smooth
domain, the global Dirichlet $W^{2,q}$ estimate gives, for every
finite $q>n$,
\[
 \|u\|_{W^{2,q}(\Omega)}\le C_q\bigl(\|f_t(u)\|_{L^q}
 +\|u\|_{L^q}\bigr)\le C'_q.
\]
By Morrey embedding, $u$ is uniformly bounded in $C^{1,\gamma}$
with $\gamma=1-n/q$. The common Lipschitz bound for $f_t$
then gives
$[f_t(u)]_{C^{0,\gamma}}\le L_{M_*}[u]_{C^{0,\gamma}}$.
Global Schauder estimates yield a common $C^{2,\gamma}$ bound.
For any fixed $\alpha'<1$, choose $q$ so that
$\alpha'<\gamma<1$, and use compact embedding in $C^{2,\alpha'}$.
If $(t_j,u_j)$ is a sequence in the solution family, a subsequence
converges to $(t,u)$ in this topology. Joint continuity of the
reaction gives $-\Delta_g u=f_t(u)$ with zero boundary trace.
The lower bound $m_*$ prevents $u\equiv0$; since $f_t\ge0$,
the strong maximum principle makes $u>0$ in $\Omega$. Thus the
family is closed and compact in the stated topology.

Hopf's boundary lemma gives $\partial_\nu u<0$ for each pair.
The normal derivative depends continuously on the pair and on
the boundary point; compactness of
$\mathscr P\times\partial\Omega$ produces $a_0>0$ with
$-\partial_\nu u\ge2a_0$ throughout. Uniform $C^2$
compactness supplies $|\nabla_g^2u|\le C_0$. On the boundary,
for every tangent vector $\tau$ the tangential Hessian identity is
\[
 \nabla_g^2u(\tau,\tau)
 =\partial_\nu u\,\mathrm{II}_{\partial\Omega}(\tau,\tau)
 \le-2a_0\kappa_0|\tau|^2.
\]
A common tubular neighborhood exists by smoothness of the fixed
domain. If the asserted negative bound on vectors perpendicular
to $\nabla_g u$ failed arbitrarily close to the boundary, one
could choose converging solutions, points approaching the boundary,
and unit tangent vectors; the limiting vector would be tangent to
$\partial\Omega$ because $\nabla_g u$ is normal there, contradicting
the displayed boundary estimate. Hence after shrinking the collar,
$|\nabla_g u|\ge a_0$ and
$\nabla_g^2u(\tau,\tau)\le-b_0|\tau|^2$ for all
$\tau\perp\nabla_g u$. Finally the common $C^1$ bound and the
zero boundary trace, integrated along inward normal geodesics,
give $u(x)\le c_1\dist(x,\partial\Omega)$.
All constants are independent of $t$ and of the particular positive
solution, as asserted.
\end{proof}

\subsection{Blow-up limits near the boundary}

\begin{lemma}\label{lem:quantBoundaryBlowup}
Let $\Omega\subset \mathbb S^n$ have smooth boundary. Suppose that
$\rho_j\downarrow0$, that $x_j\in\Omega$ and
$d_j=\dist(x_j,\partial\Omega)$ satisfy
$0\le d_j/\rho_j\le D$, and that $W_j\in C^2(\Omega)\cap C(\overline\Omega)$
solve
\[
-\Delta_g W_j=\rho_j^{-2}h_j(W_j)\quad\text{in }\Omega,
\qquad W_j=0\quad\text{on }\partial\Omega,
\]
where
\[
0\le W_j\le1,\qquad W_j(x_j)=1,\qquad h_j(0)=0,
\]
and $h_j:[0,1]\to[0,\infty)$ obey
\[
\|h_j\|_{L^\infty([0,1])}+\operatorname{Lip}_{[0,1]}h_j\le L
\]
with $L$ independent of $j$. After a subsequence, the boundary
normal-coordinate rescalings converge in
$C^{2,\alpha}_{\mathrm{loc}}(\overline{\mathbb R^n_+})$ for each fixed
$\alpha<1$, after possibly reducing the exponent, to a function $W$ with
\[
-\Delta W=h(W)\quad\text{in }\mathbb R^n_+,\qquad
W=0\quad\text{on }\partial\mathbb R^n_+,\qquad 0\le W\le1,
\]
where $h_j\to h$ uniformly on $[0,1]$ along the same subsequence.
Moreover, there is a constant $c=c(\Omega,n,D,L)>0$ such that
\begin{equation}\label{eq:quantHeight}
\frac{d_j}{\rho_j}\ge c
\end{equation}
for all sufficiently large $j$. Thus $W$ attains its maximum value one
at an interior point and is strictly positive in the half-space.
Here $C^{2,\alpha}_{\mathrm{loc}}$ up to the flat boundary means convergence
on each fixed closed half-ball of smaller radius.
\end{lemma}

\begin{proof}
Choose a nearest boundary point $z_j$. Smoothness and compactness of
$\partial\Omega$ provide a tubular radius $r_*>0$ and a finite atlas in
which all inward Fermi maps have uniformly bounded derivatives of every
fixed order. In a chart centered at $z_j$ write
\[
\mathcal F_j(\xi',s)
=\exp_{\Theta_j(\xi')}\bigl(-s\nu(\Theta_j(\xi'))\bigr),
\qquad s\ge0.
\]
The footpoint property implies $x_j=\mathcal F_j(0,d_j)$.
Define, wherever the chart is available,
\[
\widetilde W_j(y)=W_j\bigl(\mathcal F_j(\rho_jy',\rho_jy_n)\bigr),
\qquad
\widetilde g_j=\rho_j^{-2}(\mathcal F_j\circ(\rho_j\,\cdot))^*g.
\]
For each fixed $R>0$, the equation is defined on
$B_{2R}^+=B_{2R}\cap\{y_n>0\}$ for large $j$, and
\begin{equation}\label{eq:quantScaledEquation}
-\Delta_{\widetilde g_j}\widetilde W_j
=h_j(\widetilde W_j),\qquad
\widetilde W_j|_{B_{2R}\cap\{y_n=0\}}=0.
\end{equation}
The factor $\rho_j^{-2}$ in the rescaled metric gives precisely the equation in \eqref{eq:quantScaledEquation}. The coefficients of
$\widetilde g_j$ have ellipticity constants bounded above and below,
and are bounded in every $C^k(B_{2R}^+)$ independently of $j$;
they converge in those norms to the Euclidean coefficients.

Fix $R$ and $q>n$. Local boundary $W^{2,q}$ estimates on the nested
half-balls (using only zero data on their flat portions) imply
\begin{equation}\label{eq:quantW2q}
\|\widetilde W_j\|_{W^{2,q}(B_R^+)}
\le C_{R,q}\bigl(
\|\widetilde W_j\|_{L^q(B_{2R}^+)}
+\|h_j(\widetilde W_j)\|_{L^q(B_{2R}^+)}\bigr)
\le C'_{R,q}.
\end{equation}
The boundary Morrey embedding gives a uniform
$C^{1,\gamma}(\overline{B_R^+})$ bound, where
$\gamma=1-n/q>0$. Since $h_j$ has a common Lipschitz constant,
\[
[h_j(\widetilde W_j)]_{C^{0,\gamma}(\overline{B_R^+})}
\le L[\widetilde W_j]_{C^{0,\gamma}(\overline{B_R^+})}.
\]
Applying local boundary Schauder estimates to
\eqref{eq:quantScaledEquation} on a smaller half-ball gives
\[
\|\widetilde W_j\|_{C^{2,\gamma}(\overline{B_{R/2}^+})}
\le C''_{R,q}.
\]
Because $R$ and $q$ may be chosen arbitrarily large, a diagonal
subsequence converges in $C^{2,\alpha}$ on every fixed smaller
closed half-ball for any prescribed $\alpha<1$. Arzel\`a--Ascoli
applied to the equibounded and equi-Lipschitz $h_j$ also gives
$h_j\to h$ uniformly on $[0,1]$. Passing to the equation yields the
stated limit, including its zero boundary trace.

For the quantitative separation, choose $R>D+2$. The already obtained
boundary $C^1$ bound gives a constant $C_R$ independent of $j$ with
$|\partial_{y_n}\widetilde W_j|\le C_R$ on the normal segment
$\{(0,s):0\le s\le D+1\}$. Therefore
\[
1=\widetilde W_j(0,d_j/\rho_j)
-\widetilde W_j(0,0)
=\int_0^{d_j/\rho_j}\partial_{y_n}\widetilde W_j(0,s)\,ds
\le C_R\frac{d_j}{\rho_j},
\]
which proves \eqref{eq:quantHeight}. After passing to a further
subsequence, $d_j/\rho_j\to a\in[c,D]$ and $W(ae_n)=1$.
Since $-\Delta W=h(W)\ge0$ and $W\ge0$, the strong maximum principle
for the Laplacian gives $W>0$ on the connected half-space.
\end{proof}

\begin{remark}\label{rem:boundaryScalingApplications}
For the near-linear normalization in Lemma~\ref{lem:maxnorm}, put
$\rho_j=\Lambda_j^{-1/2}$ and
$h_j(s)=s^{p_j}$. For Proposition~\ref{prop:Linfty}, take
$\rho_j=M_j^{-(p_j-1)/2}$ and $h_j(s)=s^{p_j}$.
On $0\le s\le1$ these functions have a common Lipschitz constant
whenever the exponents remain in a fixed compact subset of $[1,\infty)$.
For the inhibited-reaction homotopy in the proof of
Theorem~\ref{thm:inhibited}, take
\[
\rho_j^2=\frac{M_j}{f_{t_j}(M_j)},\qquad
h_j(s)=\frac{f_{t_j}(M_js)}{f_{t_j}(M_j)}.
\]
The effective-exponent estimate
\eqref{eq:inhibitedElasticityBounds} gives
$0\le h_j(s)\le1$ and $\operatorname{Lip}h_j\le p$.
Thus the same quantitative lemma covers the two Lane--Emden
normalizations and the inhibited-reaction blow-up.
The rescaled equation and the $C^2$ limit are obtained without assuming
that the distance of the maximum point from the boundary is bounded below
\emph{before} rescaling: that lower bound is a consequence of
\eqref{eq:quantHeight}.
\end{remark}

\subsection{Smooth domain variation}

\begin{lemma}\label{lem:quantDomainSmoothness}
In the construction of Lemma~\ref{lem:domainPath}, there are constants
$R_*,q_*,Q_*,\kappa_*>0$, independent of $t\in[0,1]$, such that
\[
D_t\subset B_{R_*}(0),\qquad
q_*g_{\mathbb S^{n-1}}\le Q_{h_t}\le Q_*g_{\mathbb S^{n-1}},
\qquad \II_{\partial\Omega_t}\ge\kappa_*g_{\partial\Omega_t}.
\]
The boundaries form a smooth family of embeddings. For every
$t_0\in[0,1]$ they admit smooth, parameter-dependent diffeomorphisms
$\Psi_t:\overline\Omega_{t_0}\to\overline\Omega_t$ for $t$ close
to $t_0$, and
\begin{equation}\label{eq:quantPullbackConvergence}
\|\Psi_t^*g-\Psi_{t_0}^*g\|_{C^{k,\alpha}(\overline\Omega_{t_0})}
\longrightarrow0\quad\text{for each fixed }k\ge0,\ 0<\alpha<1.
\end{equation}
Consequently the torsion functions $u_t\circ\Psi_t$ converge in
$C^{2,\alpha}(\overline\Omega_{t_0})$ for each fixed $\alpha<1$.
\end{lemma}

\begin{proof}
Write $h_t=(1-t)\rho+t h$. The inverse Gauss map of the smooth
strictly convex body $D$ is smooth and has positive-definite
radius-of-curvature tensor
$Q_h=\nabla^2_{\mathbb S^{n-1}}h+h g_{\mathbb S^{n-1}}$.
By compactness of $\mathbb S^{n-1}$, there are $0<q\le Q$ such that
$qg\le Q_h\le Qg$. Consequently
\[
\min\{\rho,q\}\,g\le
Q_{h_t}=(1-t)\rho g+tQ_h\le\max\{\rho,Q\}\,g
\]
for all $t$. In support coordinates, the derivative of the map
\[
Y_t(\theta)=h_t(\theta)\theta+\nabla_{\mathbb S^{n-1}}h_t(\theta)
\]
on the tangent space is exactly $Q_{h_t}$, so the Euclidean
principal curvatures of $\partial D_t$ are the eigenvalues of
$Q_{h_t}^{-1}$. Thus all are bounded below by $Q_*^{-1}$ and above
by $q_*^{-1}$, and $Y_t$ is a smooth embedded hypersurface with
parameter-uniform regularity. Since the support functions and their
first derivatives are uniformly bounded, $|Y_t(\theta)|\le R_*$
for some fixed $R_*$, whence $D_t\subset B_{R_*}$.

For completeness we verify the spherical curvature transfer. Put
$W=(1+|y|^2)^{1/2}$ and
$L=(1+(y\cdot N)^2)^{1/2}$, where $N$ is the Euclidean outward normal.
The spherical outward normal to the projected boundary is
$\nu=(N,-y\cdot N)/L$. For any Euclidean boundary tangent vectors
$\tau,\sigma$,
\[
dP\,\sigma=\frac{(\sigma,0)}W
-\frac{(y,1)(y\cdot\sigma)}{W^3}.
\]
Differentiating $\nu$ in the direction $\tau$ and pairing with
$dP\,\sigma$ eliminates the derivative of $L^{-1}$, because
$\langle\nu,dP\,\sigma\rangle=0$, as well as all terms paired
with $(y,1)$, because
$\langle(D_\tau N,-y\cdot D_\tau N),(y,1)\rangle=0$.
Since $\tau\cdot N=0$, the result is
\begin{equation}\label{eq:quantCentralII}
\II_{\mathbb S^n}(dP\,\tau,dP\,\sigma)
=\frac{\langle D_\tau N,\sigma\rangle}{WL}
=\frac{\II_{\mathbb R^n}(\tau,\sigma)}{WL}.
\end{equation}
For $|y|\le R_*$, one has
$WL\le1+R_*^2$ and
$|dP\,\tau|_g^2\le|\tau|^2$. It follows that
\[
\II_{\mathbb S^n}(dP\,\tau,dP\,\tau)
\ge\frac{1}{Q_*(1+R_*^2)}|\tau|^2
\ge\frac{1}{Q_*(1+R_*^2)}|dP\,\tau|_g^2.
\]
Thus one can take $\kappa_*=[Q_*(1+R_*^2)]^{-1}$.
The domains remain geodesically convex because a Euclidean line
segment, under $P(y)=(y,1)/\sqrt{1+|y|^2}$, is mapped onto the
shorter great-circle segment joining its endpoints (up to a monotone
reparametrization). The common bound $|y|\le R_*$ also keeps the
whole family compactly contained in a single open hemisphere.

Fix $t_0$. The embeddings $P\circ Y_t$ vary smoothly in $(t,\theta)$,
and, since $P\circ Y_{t_0}$ is embedded, a tubular neighborhood of
$\partial\Omega_{t_0}$ identifies $\partial\Omega_t$, for nearby
$t$, with the normal graph of a smooth function $a_t$ over the
reference boundary, where
$\|a_t\|_{C^{k,\alpha}}\to0$ for every fixed $k,\alpha$.
In signed normal coordinates $(z,s)$ of the reference boundary choose
a fixed smooth cutoff $\chi(s)$ which equals one at $s=0$ and
vanishes outside a smaller tubular neighborhood. The map
\[
\Psi_t\bigl(\exp_z(s\nu_{t_0}(z))\bigr)
=\exp_z\bigl((s+\chi(s)a_t(z))\nu_{t_0}(z)\bigr)
\]
there, extended by the identity outside the tubular neighborhood,
is a smooth diffeomorphism for $|t-t_0|$ small, after choosing the
cutoff and parameter interval so that
$1+\partial_s(\chi(s)a_t(z))>0$. It sends the reference boundary
to the moving boundary and the reference interior to the moving
interior. It equals the identity at $t=t_0$ and satisfies
\eqref{eq:quantPullbackConvergence}.

On the fixed reference domain let
$w_t=u_t\circ\Psi_t$ and $L_t=-\Delta_{\Psi_t^*g}$.
The coefficients of $L_t$ converge in $C^{k,\alpha}$ by
\eqref{eq:quantPullbackConvergence}. On the Banach spaces
\[
X=\{w\in C^{2,\alpha}(\overline\Omega_{t_0}):w=0
\text{ on }\partial\Omega_{t_0}\},\quad
Y=C^{0,\alpha}(\overline\Omega_{t_0}),
\]
the Dirichlet operator $L_{t_0}:X\to Y$ is an isomorphism by
uniqueness (maximum principle), existence and global Schauder theory.
Furthermore $\|L_t-L_{t_0}\|_{X\to Y}\to0$. For $t$ sufficiently
close to $t_0$,
\[
L_t^{-1}=
\bigl[I+L_{t_0}^{-1}(L_t-L_{t_0})\bigr]^{-1}L_{t_0}^{-1}
\]
is a convergent Neumann-series inverse. It follows that
$w_t=L_t^{-1}1\to w_{t_0}$ in $C^{2,\alpha}$.
The finite subcover of $[0,1]$ provides uniform $C^{2,\alpha}$
bounds. Because the geometric data are smooth, bootstrapping supplies
$C^{k,\alpha}$ bounds for every fixed $k$, and in particular uniform
continuity of the boundary normal derivatives. Hopf's lemma yields
$-\partial_{\nu_t}u_t>0$ for every $t$ and every boundary point;
the parameterized boundary is compact, so its minimum is strictly
positive. The assertions follow.
\end{proof}

\begin{remark}\label{rem:quantHemisphere}
The preceding argument relies on compact hemisphere containment,
which is guaranteed by Lemma~\ref{lem:hemisphere} rather than assumed.
A useful detail in its cone proof is that if
$C=\{rx:r\ge0,\ x\in K\}$ and $x\in\mathbb S^n$, then
\[
x\in\partial C\quad\Longleftrightarrow\quad x\in\partial K.
\]
Indeed, on the annulus $1/2<|z|<2$ the radial map
$z\mapsto(|z|,z/|z|)$ is a homeomorphism, and there $C$ is the
product of a radial interval with $K$. Thus, if $C$ contained a line,
translation along that line would preserve $\partial C$ and
produce the great-circle arc in $\partial K$ used in the proof of
Lemma~\ref{lem:hemisphere}. A nonconstant great-circle arc in the
boundary has vanishing second fundamental form in its tangent
direction, contradicting strict convexity. This also justifies the
boundary identification at every nonzero ray of the cone.
\end{remark}

\subsection{Dependence of constants}

The preceding estimates enter the global arguments in different ways.
The local Hessian estimates (Proposition~\ref{prop:generalPowerSchur}
and Theorem~\ref{thm:IntrinsicSchur}) provide strictification only
after semidefiniteness is known. Near-$1$ uniqueness and eigenfunction
concavity initialize the property. Proposition~\ref{prop:continuum}
uses the compactness and the nonzero degree to obtain a connected
set; Proposition~\ref{prop:openclosed} propagates the Hessian property
along that set. In the inhibited reaction, the limiting convex reaction and
the half-space Liouville theorem of
Dupaigne--Sirakov--Souplet \cite{DupaigneSirakovSouplet2022}
and the whole-space supersolution alternative of
Armstrong--Sirakov \cite{ArmstrongSirakov2011} supply the
blow-up contradiction.
The common-collar constants for the inhibited homotopy follow
from Lemma~\ref{lem:parameterCompactness} and the intrinsic
ratio estimate, not from compactness of individual solutions alone. For torsion, the parameter is the domain deformation;
Lemma~\ref{lem:torsionDomainDependence} and
Lemma~\ref{lem:torsionCollar} supply uniformity on the entire path.

\begin{center}
\small
\begin{tabular}{@{}p{0.26\linewidth}p{0.47\linewidth}p{0.19\linewidth}@{}}
\hline
Family & Dependence of constants & Uniform in \\
\hline
Lane--Emden & $n,\Omega,p_-,p_+$ & $p$ and solution \\
Inhibited reaction & $n,\Omega,p_*,p,\lambda,\theta$ & $t$ and solution \\
Torsion deformation & fixed smooth domain path and $n$ & $t\in[0,1]$ \\
\hline
\end{tabular}
\end{center}

\section{Coordinate computations on the sphere}
\label{app:coordinates}

For completeness, we give coordinate derivations of the identities in Lemma~\ref{lem:sphereidentities}.  Let
\[
P(y)=\frac{(y,1)}{\sqrt{1+|y|^2}},
\qquad S=1+|y|^2.
\]
A calculation from the induced metric gives
\begin{align}
g_{ij}&=\frac{\delta_{ij}}S-\frac{y_iy_j}{S^2},
&g^{ij}&=S(\delta_{ij}+y_iy_j),
\label{eq:appmetric}\\
\Gamma^k_{ij}&=-\frac{\delta_i^k y_j+\delta_j^k y_i}{S}.
\label{eq:appChristoffel}
\end{align}
Consequently, for a smooth scalar function $f$,
\begin{align}
(\nabla_g^2f)_{ij}
&=\partial_{ij}f+\frac{y_j\partial_if+y_i\partial_jf}{S},
\label{eq:appHess}\\
\Delta_gf
&=S\left(\Delta_Ef+\sum_{i,j}y_iy_j\partial_{ij}f+2\sum_i y_i\partial_if\right),
\label{eq:appLap}\\
|\nabla_gf|^2
&=S\bigl(|Df|^2+(y\cdot Df)^2\bigr).
\label{eq:appGrad}
\end{align}
At $y=0$, write
\[
a_i=\partial_if(0),\quad
b_{ij}=\partial_{ij}f(0),\quad
c_{ijk}=\partial_{ijk}f(0),\quad
d_{ijkl}=\partial_{ijkl}f(0).
\]
If $T=\nabla_g^2f$, $Q=|\nabla_gf|^2$, and $H=\Delta_gf$, then differentiation gives
\begin{align}
(\nabla_kT)_{ij}(0)
&=c_{ijk}+\delta_{jk}a_i+\delta_{ik}a_j,
\label{eq:appFirst}\\
(\Delta^\nabla T)_{ij}(0)
&=\sum_kd_{ijkk}+(2n+6)b_{ij},
\label{eq:appRough}\\
(\nabla^2H)_{ij}(0)
&=\sum_kd_{ijkk}+6b_{ij}+2\delta_{ij}\sum_kb_{kk},
\label{eq:appHessLap}\\
(\nabla^2Q)_{ij}(0)
&=2\sum_kb_{ik}b_{jk}+2\sum_ka_kc_{ijk}
+2\delta_{ij}\sum_ka_k^2+2a_ia_j.
\label{eq:appHessQ}
\end{align}
Equations \eqref{eq:appFirst}--\eqref{eq:appHessQ} recover the covariant identities at the chart center.  Rotations move any point of the sphere to such a center, and the identities are tensorial.

As a radial check, on $\mathbb S^3$ the torsion function of a ball of radius $R<\pi/2$ is
\[
U(r)=\frac12\bigl(r\cot r-R\cot R\bigr),
\qquad U(0)=\frac12(1-R\cot R).
\]
Indeed,
\[
U'(r)=\frac12(\cot r-r\csc^2r),
\qquad
U''(r)=-\csc^2r+r\csc^2r\cot r,
\]
so $U''+2\cot r\,U'=-1$, $U(R)=0$, and $U'(r)/r\to-1/3$ at the center, in agreement with Lemma~\ref{lem:torsionBall}.

\end{document}